\documentclass[11pt,a4paper]{amsart}
\usepackage{amsmath}
\usepackage{amssymb}
\usepackage{amsfonts}
\usepackage{amscd}
\usepackage{amsthm}
\usepackage{xypic}
\usepackage{enumerate}
\usepackage{fancyhdr}

\input xy
\xyoption{all}

\newcommand{\Z}{\mathbb{Z}}
\newcommand{\Q}{\mathbb{Q}}
\newcommand{\C}{\mathbb{C}}
\newcommand{\N}{\mathbb{N}}
\renewcommand{\P}{\mathbb{P}}

\newcommand{\A}{\mathbb{A}}
\newcommand{\B}{\mathbb{B}}

\newcommand{\mmP}{\mathcal{P}}

\newcommand{\mmU}{\mathcal{U}}

\newcommand{\wC}{\widetilde{C}}

\newcommand{\wZ}{\widetilde{\Z}}

\newcommand{\bl}{\boldsymbol{l}}\newcommand{\bj}{\boldsymbol{j}}
\newcommand{\bi}{\boldsymbol{i}}

\newcommand{\bn}{\boldsymbol{n}}\newcommand{\bk}{\boldsymbol{k}}
\newcommand{\bun}{\boldsymbol{1}}

\newcommand{\bm}{\boldsymbol{m}}

\newcommand{\te}{\textrm}

\DeclareMathOperator*{\ch}{ch}

  \DeclareMathOperator*{\im}{im}
\DeclareMathOperator*{\length}{length}

\DeclareMathOperator*{\Sp}{Sp}
 
\DeclareMathOperator*{\tr}{tr}

\newcommand{\trn}{\tr\nolimits}
\newcommand{\Spn}{\Sp\nolimits}

\numberwithin{equation}{section}
\newtheorem{theo}[equation]{Theorem}
\newtheorem{prop}[equation]{Proposition}
\newtheorem{cor}[equation]{Corollary}
\newtheorem{lema}[equation]{Lemma}

\theoremstyle{definition}
\newtheorem{df}[equation]{Definition}
\newtheorem{obs}[equation]{Remark}
\newtheorem{ex}[equation]{Example}

\newenvironment{enumerate*}[1][{}]{\begin{itemize}\renewcommand{\labelitemi}{#1}}{\end{itemize}}

\newcommand{\vist}{\begin{flushright}
$\square$
\end{flushright}
}

\title{A chain morphism for Adams operations on rational algebraic K-theory}
\author{Elisenda Feliu}
\date{\today}
\email{efeliu@ub.edu}
\address{Gran Via de les Corts Catalanes, 585, 08007 Barcelona (Spain)}
\keywords{Adams operations, higher algebraic K-theory, chain complex of cubes}

\begin{document}

\begin{abstract} For any regular noetherian scheme $X$ and every $k\geq 1$,
we define a chain morphism $\Psi^k$ between two chain complexes whose homology
with rational coefficients is isomorphic to the algebraic $K$-groups of $X$
tensored by $\Q$. It is shown that the morphisms $\Psi^k$ induce in homology the Adams operations defined by
 Gillet and Soul{\'e} or the ones defined by Grayson.
\end{abstract}
\maketitle

\section*{Introduction}

Let $X$ be any scheme and let $\mathcal{P}(X)$ be the exact category of locally
free sheaves of finite rang over $X$. The algebraic $K$-groups of $X$, $K_n(X)$, are
defined as the Quillen K-groups of the category $\mmP(X)$, as given in \cite{Quillen}.

Several authors have equipped these groups with a $\lambda$-ring structure.
Then, the \emph{Adams operations} on each $K_n(X)$ are obtained
from the $\lambda$-operations by a universal polynomial formula on the
$\lambda$-operations. In the literature there are several
direct definitions of the Adams operations on the higher algebraic $K$-groups of a
scheme $X$.  By means of the homotopy theory of simplicial sheaves, Gillet and Soul\'e
 defined in \cite{GilletSouleFiltrations} Adams operations for any noetherian scheme of
finite Krull dimension. Grayson, in \cite{Grayson1}, constructed a
simplicial map inducing Adams operations on the K-groups of any
category endowed with a suitable tensor product, symmetric power
and exterior power. In particular, he constructed Adams operations
for the algebraic K-groups of any scheme $X$. Following the
methods of Schechtman in \cite{Sch}, Lecomte, in \cite{Lecomte},
defined Adams operations for the rational K-theory of any scheme
$X$ equipped with an ample family of invertible sheaves. They are
induced by a map in the homotopy category of infinite loop spectra.

The aim of this paper is to construct an explicit chain morphism which induces
the Adams operations on the higher algebraic $K$-groups tensored by $\Q$. Our main interest in this construction is to endow the rational higher arithmetic $K$-groups of an arithmetic variety with a (pre)-$\lambda$-ring structure, in order to pursue a higher arithmetic intersection theory program in Arakelov geometry.

At the moment, there are two different definitions for the higher arithmetic $K$-groups of an arithmetic variety, one suggested by Deligne and Soul\'e (see \cite{Soule2} $\S$III.2.3.4 and
\cite{Delignedeterminant}, Remark 5.4)  and the other given by Takeda in \cite{Takeda}. Both of them rely on an explicit representative of the Beilinson regulator. By the nature of both definitions, it is apparently necessary
to have a description of the Adams operations in algebraic K-theory in terms of a chain
morphism, compatible with the representative of the Beilinson regulator ``$\ch$'' given by Burgos and Wang in \cite{Burgos1}.
None of the explicit constructions of $\Psi^k$ known at the moment seem to be suitable for this purpose.
The chain morphism presented in this paper commutes with the morphism ``$\ch$''. In fact our definition has been highly influenced by the construction of ``$\ch$''.
The details of the application to higher arithmetic $K$-theory can be found in the author's PhD Thesis \cite{FeliuThesis}.

Consider the chain complex of cubes associated to the category
$\mmP(X)$. McCarthy in \cite{McCarthy}, showed  that the homology
groups of this complex, with rational coefficients, are isomorphic
to the rational algebraic K-groups of $X$.

We first attempted to find a homological version of Grayson's
simplicial construction using the complex of cubes, but this seems particularly difficult
from the combinatorial point of view.

The current approach is based on a simplification obtained by
using the \emph{transgressions of cubes by affine or projective
lines}, at the price of having to reduce to regular noetherian
schemes due to the fact that homotopy invariance or the Dold-Thom isomorphism for $K$-theory are required. This was Burgos and Wang's idea \cite{Burgos1} for the definition of a chain morphism representing Beilinson's regulator.

With this strategy, we first assign to a cube on $X$ a collection
of cubes defined either on $X\times (\P^1)^*$ or on $X\times
(\A^1)^*$, which have the property of being split in all
directions (and which we call \emph{split cubes}).  This gives a morphism  called the
transgression morphism ({\bf Proposition \ref{trans2}}).

Then, by a purely combinatorial formula on the Adams operations of
locally free sheaves, we give a formula for the Adams operations
on split cubes ({\bf Corollary \ref{adamssplitcor}}). The key
point is to use Gillet's idea, as presented by Grayson, of considering the secondary Euler
characteristic class of the Koszul complex associated to a locally
free sheaf of finite rank.

The composition of the transgression morphism with the Adams
operations for split cubes gives a chain morphism representing the
Adams operations for any regular noe\-the\-rian scheme of finite
Krull dimension ({\bf Theorem \ref{myadamsgood}}). The fact that our construction induces indeed
the Adams operations defined by Gillet and Soul\'e in \cite{GilletSouleFiltrations} and the ones defined by Grayson in \cite{Grayson1} follows
from a general result on the comparison of morphisms from algebraic $K$-theory to itself, given in \cite{Feliu1}.

The two constructions, with projective lines or with affine lines,
are completely analogous. One may choose the more suitable one in
each particular case. For instance, to define Adams operations on
the $K$-groups of a regular ring $R$, one may
 consider the definition with affine lines so as to remain in the category of affine schemes.
 On the other hand, if for instance our category of
schemes is the category of projective regular schemes, then the
construction with projective lines may be the appropriate one.

\medskip
The paper is organized as follows. In the first section, we introduce the
notation for multi-indices and (co)chain complexes. The complex of cubes is defined and a \emph{normalized} version, in the style of the normalized complex
associated to a cubical abelian group, is introduced.
In the next section we define Adams operations for \emph{split cubes}, that
is, for cubes which are split in all directions, by means of a
combinatorial formula on the Adams operations of locally free
sheaves of finite rank.
In the third section, the \emph{transgression morphism} is defined. We assign to every cube of locally free sheaves on $X$, a
collection of split cubes defined either on $X\times (\P^1)^*$
or on $X\times (\A^1)^*$.
Finally, in the last section we summarize the constructions provided in the previous sections so as to give
a representative of the Adams operations for regular noetherian schemes of finite Krull dimension.
It is shown that our
construction induces the Adams operations defined by Gillet and
Soul{\'e} in \cite{GilletSouleFiltrations}.

\medskip
\emph{Aknowledgement.} First of all, I would like to thank my thesis advisor,
Jos{\'e} Ignacio Burgos Gil, for his always good ideas and constant support
during the elaboration of this paper. I am very grateful to Damian R\"ossler who
suggested to me that it should be possible to use the transgressions by
projective lines to construct Adams operations for cubes. Finally, I would like
to thank Jos{\'e} Gil for helpful discussions about universes.

\section{The chain complex of cubes}

\subsection{Notation for multi-indices}\label{multiindices} We give here some notations on multi-indices
that will be used in the sequel.

Let $\mathfrak{I}$ be the set of all multi-indices of finite
length, i.e.
$$\mathfrak{I}=\{\boldsymbol{i}=(i_1,\dots,i_n)\in \N^n,\ n\in \N\}=
\bigcup_{k>0}\N^k.$$ For every $m\geq 0$, consider the set
$[0,m]:= \{0,\dots,m\}$. If $a\in [0,m]$ and
$l=1,\dots,n$, let $a_l\in [0,m]^n$ be the multi-index
$$(0,\dots,0,a,0,\dots,0),$$ that is, the multi-index where the
only non-zero entry is $a$ in the $l$-th position. We write
$\bun=(1,\dots,1)$ and more generally, if $r_1\leq r_2$, we define
$\bun_{r_1}^{r_2}$  to be the
multi-index with
$$(\bun_{r_1}^{r_2})_{i} =\left\{ \begin{array}{ll}
1 & \textrm{if }\ r_1 \leq i \leq r_2, \\ 0 & \textrm{otherwise}.
\end{array}\right.  $$

\begin{df} Let $\bi,\bj\in \N^n$. We fix the following notations for multi-indices:
\begin{enumerate}[(1)]
\item The \emph{length} of $\bi$
is the integer $\length(\bi):=n$. \item The \emph{characteristic} of $\bi$ is the multi-index
$\nu(\bi)\in \{0,1\}^n$,  defined by
$$\nu(\bi)_j = \left\{ \begin{array}{ll}
0 & \textrm{if }i_j=0, \\
1 & \textrm{otherwise}.
\end{array}\right.$$
\item The \emph{norm}  of $\bi$ is
defined by $|\bi|=i_1+\cdots +i_n$.
 If $1\leq l \leq n$, we denote
$|\bi|_l=i_1+\cdots +i_l$. \item \emph{Orders on the set of
multi-indices: }
\begin{enumerate*}[$\rhd$] \item We write  $\bi \geq \bj$, if
for all $r$, $i_r \geq j_r$. Otherwise we write $\bi \ngeq \bj$.
\item We denote by $\preceq$ the \emph{lexicographic order} on multi-indices. By
$\bi \prec \bj$ we mean $\bi \preceq
\bj$ and $\bi\neq \bj$.
\end{enumerate*}
 \item Let $1\leq l \leq n$ and $m\in \N$. Then, we define
$$
\begin{array}{lrcl}
\emph{Faces:} & \partial_l(\bi)&:=&(i_1,\dots,\hat{i_l},\dots,i_n). \\
\emph{Degeneracies:}  &  s_l^m(\bi)&:=& (i_1,\dots,i_{l-1},m,i_{l},\dots,i_n). \\
\emph{Substitution:}  &
\sigma_l^m(\bi)&:=&
s_l^m\partial_l(\bi)=(i_1,\dots,i_{l-1},m,i_{l+1},\dots,i_n).
 \end{array}
$$
In general, for any $\bl=(l_1,\dots,l_s)$ with $1\leq l_1 < \dots<
l_s \leq n$ and $\mathbf{m}=(m_1,\dots,m_s)\in \N^s$, we write
{\small $$\partial_{\bl}(\bi)=\partial_{l_1}\dots \partial_{l_s}(\bi),
\quad s_{\bl}^{\mathbf{m}}(\bi)=s_{l_1}^{m_1}\dots
s_{l_s}^{m_s}(\bi),\quad\textrm{and}\quad
\sigma_{\bl}^{\mathbf{m}}(\bi)=\sigma_{l_1}^{m_1}\dots
\sigma_{l_s}^{m_s}(\bi).$$ }\item If $\length(\bi)=l$ and
$\length(\bj)=r$, the \emph{concatenation}  of $\bi$ and $\bj$ is the multi-index of length
$l+r$ given by
 $$\bi \bj=(i_1,\dots,i_l,j_{1},\dots,j_r). $$
 \item Assume that $\bi
\in \{0,1\}^n$. The
\emph{complementary multi-index} of $\bi$ is the multi-index
$\bi^c :=\bun - \bi$, i.e.
$$(\bi^c)_r = \left\{ \begin{array}{ll}
0 & \textrm{if }i_r=1, \\
1 & \textrm{if }i_r=0.
\end{array}\right.$$
\item Assume that $\bi,\bj\in\{0,1\}^n$. We define their
\emph{intersection}   by
$$\bi \cap \bj = (i_1 \cdot j_1,\dots,i_n\cdot j_n),$$
and their   \emph{union} $\bi \cup
\bj$ by $$(\bi \cup \bj)_r = \max\{i_r,j_r\}.$$
 \end{enumerate}
\end{df}

\subsection{Iterated (co)chain complexes } Let $\mathcal{U}$ be some universe
(see \cite{SGA4}) and let $\mmP$ be a small additive category in $\mathcal{U}$,
with fixed zero object $0$.
\begin{df}
\begin{enumerate}[(i)]
\item A \emph{$k$-iterated cochain complex} $C^*=(C^{*},d^1,\dots,d^k)$ over
$\mmP$ is a $k$-graded object together with $k$ endomorphisms $d^1,\dots,d^k$
of multi-degrees $1_1,\dots,1_k$, respectively, such that for all $i,j$,
$d^id^i=0$ and $d^id^j=d^jd^i$. The endomorphism $d^i$ is called the $i$-th
\emph{differential} of $C^*$. \item  A \emph{$k$-iterated chain complex}
$C_*=(C_{*},d_1,\dots,d_k)$ over $\mmP$ is a $k$-graded object together with $k$
endomorphisms $d_1,\dots,d_k$,
  of multi-degrees $-1_1,\dots,-1_k$ respectively,
such that for all $i,j$, $d_id_i=0$ and $d_id_j=d_jd_i$. The endomorphism $d_i$
is called the $i$-th \emph{differential} of $C_*$.\item A \emph{(co)chain
morphism} is a collection of morphisms commuting with the differentials.
\end{enumerate}
\end{df}

If $C^*$ is a $k$-iterated cochain complex and $\bi$ a multi-index of length
$k-1$, then $C^{s_l^*(\bi)}$ is a cochain complex. In this way,  if $P$ is a
property of cochain complexes, we say that \emph{$C^*$ satisfies the property $P$
in the $l$-th direction}, if for all multi-indices $\bi$ of length $k-1$, the
cochain complex $C^{s_l^*(\bi)}$ satisfies $P$.

Let $C^*$ be a cochain complex. We will mainly refer to the two following properties of cochain complexes:
\begin{enumerate}[(i)]
\item The complex $C^*$ has \emph{finite length} if there exists $l_1<l_2$ such
that
$$ C^n=0,\qquad \textrm{ for }n<l_1,\textrm{ and }n>l_2.$$ In this case, the difference $l_2-l_1$ is called
the \emph{length} of $C$.
 \item The complex $C^*$ is \emph{acyclic}, if $H^n(C)=0$ for all $n$.
\end{enumerate}

\begin{df} Let $(B^*,d^1,d^2)$ be a $2$-iterated cochain complex. The \emph{simple complex} of $B^*$
is the  cochain complex whose graded groups are
$$B^n:= \bigoplus_{r+s=n} B^{r,s},$$
and whose differential is
\begin{eqnarray*}
B^{r,s} &\xrightarrow{d} &  B^{r+1,s}\oplus B^{r,s+1}
\\ b &\mapsto & d^1(b)+(-1)^{r}d^2(b).
 \end{eqnarray*}
\end{df}
Observe that if $B$ is acyclic in one direction, then so is the simple complex.

\begin{ex}[Tensor product] Assume that in the category $\mmP$ there is a notion of tensor product.
In our applications, $\mmP$ will be the category of abelian groups or the category of locally free sheaves on a scheme.
Let $(A^*,d_A)$ and $(B^*,d_{B})$ be two cochain complexes. The
tensor product $(A \otimes B)^*$ is the $2$-iterated cochain
complex with
$$(A\otimes B)^{n,m} = A^{n}\otimes B^{m},$$
and differentials $(d_A\otimes id_{B},id_{A}\otimes d_{B})$. By
abuse of notation, the associated simple complex will also be
denoted by $(A \otimes B)^*$.
\end{ex}

\subsection{The chain complex of iterated cochain complexes}\label{iterated}
Let $\mmP$ be a $\mmU$-small abelian category.
We denote by $IC_n(\mmP)$ the set of
$n$-iterated cochain complexes over $\mmP$, concentrated in non-negative
degrees, of finite length and acyclic in all directions. Let $\Z IC_n(\mmP)$ be
the free abelian group generated by $IC_n(\mmP)$. Then,
$$\Z IC_*(\mmP)=\bigoplus_{n\geq 0} \Z IC_n(\mmP)$$
is a graded abelian group, which can be made into a chain complex in the following way. For every
$\bl=(l_1,\dots,l_n)$, we denote by $IC_n^{\bl}(\mmP)  \subseteq IC_n(\mmP)$
the set of $n$-iterated cochain complexes of length $l_i$ in the $i$-th
direction.
\begin{df}\label{faces2} Let $A^*\in IC_n^{\bl}(\mmP)$. For every $i=1,\dots,n$ and $j\in [0,l_i]$,
the $(n-1)$-iterated cochain complex $\partial_i^{j}(A)^*$ is defined by
$$\partial_i^{j}(A)^{\bm}:=A^{s_i^j(\bm)}\in IC_{n-1}(\mmP)\qquad \forall \bm.$$
It is called the \emph{$j$-th face of $A^*$ in the $i$-th direction}. If $j> l_i$, we set
$\partial_i^{j}(A)^{\bm}:=0.$
\end{df}
It follows from the definition that  for all $j\in [0,l_{i}]$ and $k \in [0,l_{r}]$,
\begin{equation}\label{commutfaces}
\partial_{i}^{j}\partial^{k}_{r}=
\partial^{k}_{r-1}\partial^{j}_{i},\quad \textrm{if }i\leq r.
\end{equation}
Then, there is a well-defined group morphism
\begin{eqnarray*}
\Z IC_n(\mmP) & \xrightarrow{d}&  \Z IC_{n-1}(\mmP) \\
A^* & \mapsto & \sum_{i=1}^{n}\sum_{j\geq 0} (-1)^{i+j}\partial_i^j(A)^*.
\end{eqnarray*}
Since $d^2=0$, the pair  $(\Z IC_*(\mmP),d)$ is a chain complex. It is called
\emph{the chain complex of iterated cochain complexes}.

\begin{obs}
Observe that we have obtained a \emph{chain complex} whose $n$-graded piece is
generated by $n$-iterated \emph{cochain complexes}. We will try to be very precise
on this duality, so as not to confuse the reader.
\end{obs}

\subsection{The chain complex of cubes}\label{cubessection}
We are interested in the chain complex of iterated cochain complexes
obtained restricting to the iterated cochain complexes of length $2$ in all
directions. We write for simplicity,
$$C_n(\mmP)= IC_n^{\boldsymbol{2}}(\mmP)\quad \textrm{and}\quad \Z C_n(\mmP)=
\Z IC_n^{\boldsymbol{2}}(\mmP).$$ The differential of $\Z IC_*(\mmP)$ induces a
differential on $\Z C_*(\mmP)=\bigoplus_n \Z C_n(\mmP),$ making the inclusion
$\Z C_*(\mmP)\hookrightarrow \Z IC_*(\mmP)$ a chain morphism.
An
element of $C_n(\mmP)$ is called an \emph{$n$-cube.}

\begin{obs}\label{cubes22}\label{cubes2}
Let
$$ \varepsilon: 0\rightarrow E_0 \rightarrow E_1 \rightarrow E_2 \rightarrow 0$$
be an exact sequence of $(n-1)$-cubes. That is, for every $\bj\in
\{0,1,2\}^{n-1}$, the sequence
$$0\rightarrow E_0^{\bj}\rightarrow E_1^{\bj} \rightarrow E_2^{\bj}\rightarrow 0 $$
is exact. Then, for all $i=1,\dots,n$, there is an $n$-cube
$\widetilde{E}$, with
$$\partial_i^j \widetilde{E}= E_j.$$ This cube is called the
\emph{cube obtained from $\varepsilon$ along the $i$-th
direction}.
\end{obs}

\begin{df}
For every $i=1,\dots,n$ and $j=0,1$, one   defines
\emph{degeneracies}
$$s_i^j:\Z C_{n-1}(\mmP)\rightarrow \Z C_{n}(\mmP),$$
by setting for every $E\in C_{n-1}(\mmP)$,
$$s_i^j(E)_{\bj}=\left\{ \begin{array}{ll} 0 & j_i\neq j,j+1 \\
E_{\partial_i(\bj)}& j_i=j,j+1.
\end{array}  \right.   $$
That is, $s_i^j(E)$ is the $n$-cube obtained from the exact
sequences of $n$-cubes
$$
\begin{array}{ccc}
0\rightarrow E \xrightarrow{=} E \rightarrow 0 \rightarrow 0, & \textrm{if }
j=0, \\
0\rightarrow 0 \rightarrow E \xrightarrow{=} E \rightarrow 0, &
\textrm{if } j=1,
\end{array}
$$
along the $i$-th direction. An element $F\in C_n(\mmP)$ is called
\emph{degenerate}  if for some $i$ and
$j$, $F \in \im s_i^j.$
\end{df}
For any $k,l\in \{0,1,2\}$ and for all
$u,v\in \{0,1\}$, the following identities are satisfied:
\begin{equation}\label{identities3}
 \begin{array}{rcl} \partial_i^l \partial_{j}^k &=& \left\{ \begin{array}{ll}
\partial_{j}^k \partial_{i+1}^l & \textrm{if }j\leq i, \\
\partial_{j-1}^k \partial_{i}^l & \textrm{if }j>
i.
\end{array}\right. \\
 \partial_{i}^0 s_i^0 & =&  \partial_{i}^1 s_i^0 = id, \quad \partial_{i}^1 s_i^1
= \partial_{i}^2 s_i^1 = id, \quad \partial_{i}^2 s_i^0 = \partial_{i}^0 s_i^1
= 0, \\  \partial_i^l s_j^u &=& \left\{
\begin{array}{ll} s_{j}^u \partial_{i-1}^l & \textrm{if }j< i, \\ s_{j-1}^u
\partial_{i}^l
& \textrm{if }j> i.
\end{array}\right. \\
s_i^us_j^v & = & s_{j+1}^vs_{i}^u \textrm{ if }j\geq i.
\end{array}
\end{equation}

 Let
$$\Z D_n(\mmP)=\sum_{i=1}^{n} s_i^0(\Z C_{n-1}(\mmP))+ s_i^1(\Z
C_{n-1}(\mmP))\subset \Z C_n(\mmP).$$ Since the differential of a
degenerate cube is also degenerate, the differential of $\Z
C_*(\mmP)$ induces a differential on $\Z D_*(\mmP)$ making the
inclusion arrow $\Z D_*(\mmP) \hookrightarrow \Z C_*(\mmP)$ a chain morphism. The quotient complex
$$\wZ C_*(\mmP)=
\Z C_*(\mmP)/\Z D_*(\mmP)$$ is called the \emph{chain complex of
cubes} in $\mmP$. Nevertheless, by abuse of language, the complex $\Z C_*(\mmP)$ is usually referred as to the chain complex of cubes as well.

\begin{prop}\label{cubes}(McCarthy) Let $\mmP$ be a small abelian category and
let $K_n(\mmP)$ denote the Quillen algebraic K-groups of $\mmP$. Then, for all
$n\geq 0$, there is an isomorphism
$$H_n(\wZ C(\mmP),\Q)\cong K_n(\mmP)\otimes \Q. $$
\end{prop}
\begin{proof}
See \cite{McCarthy}.
\end{proof}

\subsection{The normalized complex of cubes}\label{sectionnormalized}
Let $\mathcal{P}$ be a small exact category in some universe $\mmU$. In this
section, we show that there is a normalized complex  for the complex of cubes,
in the style of the normalized complex associated to a simplicial or cubical abelian group. That is, we construct a complex
$NC_*(\mmP)\subset \Z C_*(\mmP)$, which maps isomorphically to $\wZ C_*(\mmP)$.

\begin{prop}\label{N} Let  $NC_*(\mmP)  \subset \Z C_*(\mmP)$ be
any of the following complexes:
$$N_nC(\mmP) = \left\{ \begin{array}{l} \bigcap_{i=1}^{n} \ker \partial_i^0 \cap
\bigcap_{i=1}^{n} \ker \partial_i^2, \vspace{0.25cm} \\  \bigcap_{i=1}^{n} \ker
\partial_i^0 \cap \bigcap_{i=1}^{n} \ker (\partial_i^1-\partial_i^0)=
\bigcap_{i=1}^{n} \ker
\partial_i^0 \cap  \bigcap_{i=1}^{n} \ker
\partial_i^1,\vspace{0.25cm} \\ \bigcap_{i=1}^{n} \ker (\partial_i^1-\partial_i^2) \cap
\bigcap_{i=1}^{n} \ker \partial_i^2=  \bigcap_{i=1}^{n} \ker \partial_i^1
\cap \bigcap_{i=1}^{n} \ker \partial_i^2, \vspace{0.25cm} \\
\bigcap_{i=1}^{n} \ker (\partial_i^1-\partial_i^2) \cap  \bigcap_{i=1}^{n} \ker
(\partial_i^0-\partial_i^1). \end{array} \right. $$
  Then, the composition
$$NC_*(\mmP) \hookrightarrow \Z C_*(\mmP) \twoheadrightarrow \wZ C_*(\mmP)=\Z C_*(\mmP)/\Z D_*(\mmP)$$
is an isomorphism of chain complexes.
\end{prop}

\begin{proof}
We will see that the complex of cubes can be obtained by associating two
different cubical structures to the collection of abelian groups $\{\Z C_n(\mmP)\}_n$.

We start by recalling the definitions and results on cubical abelian groups that we need. Given a cubical abelian group $C_{\cdot}$, with face maps denoted by $\delta_i^j$ and degeneracy maps by $\sigma_i$, the chain complex associated to $C_{\cdot}$, $C_*$, is the chain complex whose $n$-th graded piece is
$C_n$  and whose differential $\delta:C_n\rightarrow C_{n-1}$ is given by $\delta= \sum_{i=1}^n \sum_{j=0,1}(-1)^{i+j}\delta_i^j.$
Let $D_n\subset C_n$ be the subgroup of \emph{degenerate elements} of $C_n$, i.e. the elements that lie
in the image of $\sigma_i$ for some $i$. The quotient
$\widetilde{C}_*:=C_*/D_*$ is a chain complex, whose differential is induced by $\delta$.
For $l=0$ or $1$, the \emph{normalized chain complex} associated to $C_{\cdot}$, $N^lC_*$, is the chain complex whose $n$-th graded group is
$$N^lC_n := \bigcap_{i=1}^{n} \ker \delta_i^l,$$  and whose
differential is the one induced by the inclusion $N^lC_n\subset
C_n$.
A well-known result states that for any cubical abelian group  $C_{\cdot}$, there is a decomposition of chain complexes
$C_*= N^lC_* \oplus D_*. $ As a consequence, we obtain that the
composition
\begin{equation}\label{norm}
\phi: N^lC_* \hookrightarrow C_* \twoheadrightarrow \widetilde{C}_*
\end{equation}
is an isomorphism of chain complexes.

In our situation, we associate two
different cubical structures to the collection of abelian groups $\{\Z C_n(\mmP)\}_n$, we apply twice the normalized construction to $\Z C_{\cdot}(\mmP)=\{\Z C_n(\mmP)\}_n$ and finally we obtain a
subcomplex $NC_*(\mmP)\subset \Z C_*(\mmP)$ which is isomorphic to $\Z
C_*(\mmP)/\Z D_*(\mmP)$.

The two different cubical structures  of
$\Z C_{\cdot}(\mmP)$ are given as follows.
\begin{enumerate*}[$\blacktriangleright$] \item For the \emph{first structure} consider
$$\tilde{\partial}_i^0=\partial_i^0,\quad \tilde{\partial}_i^1=\partial_i^1-\partial_i^2,
\quad\textrm{and} \quad \tilde{s}_i=s_i^0.$$ \item For the \emph{second
structure} consider
$$\tilde{\partial}_i^0=\partial_i^2,\quad \tilde{\partial}_i^1=\partial_i^1-\partial_i^0,
\quad\textrm{and} \quad \tilde{s}_i=s_i^1.$$
\end{enumerate*}
By the identities \eqref{identities3}, both collections of faces
and degeneracies satisfy the identities of a cubical structure on
$\Z C_{\cdot}(\mmP)$. Moreover, the differential of $\Z C_*(\mmP)$
induced by both structures is exactly the differential of the
complex of cubes.

By the first structure, we obtain an isomorphism of chain complexes
$$ N^1C_*(\mmP) \hookrightarrow \Z C_*(\mmP) \twoheadrightarrow \Z C_*(\mmP)/\Z D^1_*(\mmP)$$
where
$$
 N^1
C_*(\mmP)= \bigcap_{i=1}^{n}\ker \partial_i^0 \ \textrm{ or }\
\bigcap_{i=1}^{n}\ker (\partial_i^1-\partial_i^2),\quad\textrm{and}\quad
 \Z D_n^1(\mmP) = \sum_{i=1}^n \im s_i^0.
$$
The reader can check that the second structure induces a cubical
structure on $N^1C_{\cdot}(\mmP)$ and on $\Z C_{\cdot}(\mmP)/\Z
D^1_{\cdot}(\mmP)$ compatible with the map $N^1C_{\cdot}(\mmP)
\rightarrow \Z C_{\cdot}(\mmP)/\Z D^1_{\cdot}(\mmP)$. Therefore,
there is an isomorphism of complexes
$$N^2N^1C_*(\mmP) \hookrightarrow N^2 C_*(\mmP) \twoheadrightarrow N^2(\Z C_*(\mmP)/\Z D^1_*(\mmP)). $$
Applying \eqref{norm} to $\Z
C_{\cdot}(\mmP)/\Z D^1_{\cdot}(\mmP)$, we obtain an isomorphism of
complexes
$$N^2(\Z C_*(\mmP)/\Z D^1_*(\mmP)) \hookrightarrow  \Z C_*(\mmP)/\Z D^1_*(\mmP) \twoheadrightarrow
\frac{\Z C_*(\mmP)/\Z D^1_*(\mmP)}{\sum_{i=1}^{*}\im s_*^1}.$$
Since for every $n$, $\sum_{i=1}^n\im s_i^0 \cap \sum_{i=1}^n\im
s_i^1=\{0\}$, we obtain that
$$\frac{\Z C_*(\mmP)/\Z D^1_*(\mmP)}{\sum_{i=1}^{*}\im s_*^1}=\Z C_*(\mmP)/\Z D_*(\mmP).$$

Hence, $NC_*(\mmP)=N^2N^1C_*(\mmP)$ is isomorphic to $\wZ C_*(\mmP)$. The four
candidates in the statement of the proposition appear combining the two options
for the normalized complex, for every structure.
\end{proof}

\begin{obs}
There are actually two other possible cubical structures on $\Z
C_{\cdot}(\mmP)$. One can consider the structure with
$$\tilde{\partial}_i^0=\partial_i^0+\partial_i^2,\quad
\tilde{\partial}_i^1=\partial_i^1\quad\textrm{and}\quad\tilde{s}_i=s_i^0\quad
\textrm{or}\quad s_i^1.$$ Therefore, further normalized complexes
are obtained. As long as we consider one cubical structure with
$\tilde{s}_i=s_i^0$ and another cubical structure with
$\tilde{s}_i=s_i^1$, we obtain different normalized complexes
associated to $\Z C_*(\mmP)$.
\end{obs}

 We fix, from now on,
the \emph{normalized chain complex} $N C_*(\mmP)$ to be the one
with $n$-th graded piece given by
$$N_n C(\mmP) := \bigcap_{i=1}^{n} \ker\partial_i^0\cap \bigcap_{i=1}^{n}
\ker \partial_i^2\subset \Z C_n(\mmP),$$ and differential induced
by the differential of $\Z C_*(\mmP)$.

As a consequence of proposition \ref{cubes} and proposition \ref{N} we obtain the following corollary.

\begin{cor}\label{cubes1} Let $\mmP$ be a small abelian category and
let $K_n(\mmP)$ denote the Quillen algebraic K-groups of $\mmP$. Then, for all
$n\geq 0$, there is an isomorphism
$$H_n(NC_*(\mmP),\Q)\cong K_n(\mmP)\otimes \Q. $$
\end{cor}

\begin{obs}
Let $X$ be a scheme and let $\mmP=\mmP(X)$
be the category of locally free sheaves of finite rank on $X$.
Fix a universe $\mmU$ so that $\mmP(X)$ is $\mathcal{U}$-small for all $X$.
We will denote the complexes $\Z C_*(\mmP), N_*C(\mmP),\Z IC_*(\mmP),\dots$ and so on simply by
$\Z C_*(X)$, $N_*C(X)$, $\Z IC_*(X),\dots$.
\end{obs}

\section{Adams operations for split cubes}\label{adamssplit}
Let $X$ be any scheme. In this section, for every $k\geq 1$, we
construct a chain morphism  $\Psi^k$ from the complex of split
cubes to the complex of cubes on $X$.

We divide the construction into three steps. We first construct
the chain complex of split cubes on $X$, $(\Z \Spn_*(X),d)$. We
then define an intermediate chain complex $(\Z G^k(X)_*,d_s)$ and
a chain morphism
$$\Z G^k(X)_*\xrightarrow{\mu \circ \varphi} \Z C_*(X).$$
Finally,  for every $n$, we construct a morphism
$$\Psi^k:\Z \Spn_n(X) \rightarrow \Z G^k(X)_n.$$
Its composition with $\mu\circ\varphi$,
$$\mu\circ\varphi\circ\Psi^k: \Z \Spn_*(X) \rightarrow  \Z C_*(X),$$
gives the definition of the Adams operations over split cubes.

Let $X$ be a scheme and let
$\mmP=\mmP(X)$ be the category of locally free sheaves of finite
rank on $X$. Recall that the notation on multi-indices was
introduced in section \ref{multiindices}.

\subsection{Split cubes}\label{split}
We introduce here the \emph{complex of split cubes}, which
plays a key role in the definition of the Adams operations for
arbitrary cubes. Roughly speaking, split cubes are the cubes which are
split in all directions.

For every $\bj=(j_1,\dots,j_n)\in \{0,1,2\}^n$, let $u_1<\dots <
u_{s(\bj)}$ be the indices such that $j_{u_i}=1$ and let
\begin{equation}
u(\bj)=(u_1,\dots,u_{s(\bj)}). \end{equation}
 Observe that $s(\bj)$
is the length of $u(\bj)$.

\begin{df}
Let $\{ E^{\bi}\}_{\bi\in \{0,2\}^n}$ be a collection of locally free sheaves
on $X$, indexed by $\{0,2\}^n$. Let $[E^{\bi}]_{\bi\in \{0,2\}^n}$ be the
$n$-cube given by:
\begin{enumerate*}[$\rhd$]
\item The $\bj$-component is
$$\bigoplus_{\bm\in\{0,2\}^{s(\bj)} } E^{\sigma_{u(\bj)}^{\bm}(\bj)}, \quad \bj\in \{0,1,2\}^n.$$
\item
The morphisms are compositions of the following canonical morphisms:
$$\begin{array}{rclcrcl}
A \oplus B & \twoheadrightarrow & A, & \qquad &  A\oplus B & \xrightarrow{\cong} & B \oplus A,\\
A & \hookrightarrow &  A \oplus B,  & \qquad & A \oplus (B\oplus C) & \xrightarrow{\cong}
 & (A\oplus B) \oplus C.
\end{array}$$
\end{enumerate*}
 \nopagebreak[0]
 An $n$-cube of this form is called a \emph{direct sum $n$-cube}.
\end{df}

\begin{obs}
In the previous definition, the direct sum is taken in the
lexicographic order on $\{0,2\}^{s(\bj)}$.
\end{obs}

Observe that, if $\bj\in \{0,2\}^n$, then the $\bj$-component of
$[E^{\bi}]_{\bi\in \{0,2\}^n}$ is exactly $E^{\bj}$. Hence, this
$n$-cube has at the ``corners'' the given collection of objects,
and we fill the ``interior'' with the appropriate direct sums.

\begin{ex}
For $n=1$, the $1$-cube $[E^0,E^2]$ is the exact sequence
$$ E^0\rightarrow E^0\oplus E^2 \rightarrow E^2.$$
\end{ex}
\begin{ex}
For $n=2$, if  $E^{00},E^{02},E^{20},E^{22}$ are locally free sheaves on $X$,
then the $2$-cube  $\left[\begin{array}{cc}E^{00} & E^{02} \\
E^{20} & E^{22}\end{array}\right]$ is the $2$-cube {\small
$$\xymatrix{ E^{00} \ar[r] \ar[d] & E^{00}\oplus E^{02} \ar[r]
\ar[d] & E^{02}
  \ar[d] \\
 E^{00}\oplus E^{20} \ar[r] \ar[d] & E^{00}\oplus E^{02}\oplus E^{20}
     \oplus E^{22} \ar[r] \ar[d] & E^{02} \oplus E^{22}
  \ar[d] \\
 E^{20}\ar[r] & E^{20}\oplus E^{22} \ar[r] & E^{22}.}$$ }
\end{ex}

\begin{df}\label{splitcube}
\begin{enumerate*}[$\blacktriangleright$] \item Let $E$ be an $n$-cube. The \emph{direct sum $n$-cube associated
to $E$}, $\Sp(E)$, is the $n$-cube
$$\Sp(E):=[E^{\bj}]_{\bj\in \{0,2\}^n}.$$  \item A \emph{split $n$-cube} is
a couple $(E,f)$, where $E$ is an $n$-cube and $f:\Sp(E)\rightarrow E$ is an
isomorphism of $n$-cubes such that $f^{\bj}=id$ if $\bj\in \{0,2\}^{n}$.
The morphism $f$ is called the \emph{splitting} of $(E,f)$.
\item Let
$$\Z \Sp\nolimits _n(X):= \Z \{\textrm{split } n-\textrm{cubes on } X \},$$
and let $\Z \Sp_*(X)=\bigoplus_n \Z \Sp_n(X)$.
\end{enumerate*}
\end{df}

\begin{ex}\label{splitn2}
For $n=2$, a split cube $(E,f)$ consists of a $2$-cube $E$ together with an isomorphism { \small $$
\begin{array}{c}\xymatrix@C=16pt{ E^{00} \ar[r] \ar[d] & E^{00}\oplus
E^{02} \ar[r] \ar[d] & E^{02}
  \ar[d] \\
 E^{00}\oplus E^{20} \ar[r] \ar[d] & E^{00}\oplus  E^{02} \oplus E^{20}
\oplus E^{22} \ar[r] \ar[d] & E^{02} \oplus E^{22}
  \ar[d] \\
 E^{20}\ar[r] & E^{20}\oplus E^{22} \ar[r] & E^{22}}
\end{array}  \xrightarrow{f} \begin{array}{c}
 \xymatrix@C=16pt{ E^{00} \ar[r] \ar[d] & E^{01} \ar[r]
\ar[d] & E^{02}
  \ar[d] \\
 E^{10} \ar[r] \ar[d] & E^{11} \ar[r] \ar[d] & E^{12}
  \ar[d] \\
 E^{20}\ar[r] & E^{21}  \ar[r] & E^{22},}\end{array}$$ }
 which is the identity at the ``corners''.
 \end{ex}

We endow $\Z \Sp_*(X)$ with a chain complex structure. That is, we define a differential morphism
$$\Z \Sp\nolimits_n(X)\rightarrow \Z \Sp\nolimits_{n-1}(X).$$
Let $E$ be an arbitrary $n$-cube. Observe that if $j=0,2$, then,   for all $l=1,\dots,n$,
$$\partial_l^j\Sp(E)=\Sp(\partial_l^jE).  $$
 Therefore, if $(E,f)$ is a split $n$-cube,
$$\partial_l^j(E,f):=(\partial_l^jE,\partial_l^j f)$$ is a split $(n-1)$-cube. By contrast, in general
\begin{equation}\label{partial1}\partial_l^1\Sp(E)\neq
\Sp(\partial_l^1E).\end{equation} However, if $E$ is a split
$n$-cube, $\partial_l^1E$ is also isomorphic to
$\Sp(\partial_l^1E)$, i.e. it is also split.

In order to illustrate the forthcoming definition, we will start
by defining the face $\partial_l^1(E,f)=(\partial_l^1E,\hat{f})$
for $n=2$. Let $(E,f)$ be a split $2$-cube as in example \ref{splitn2}.
Then,
$$\partial_1^1(E)=E^{10}\rightarrow E^{11} \rightarrow E^{12} \quad \textrm{and}\quad \Sp(\partial_1^1E)=E^{10}\rightarrow E^{10}\oplus E^{12} \rightarrow E^{12}.$$
We define the morphism $\hat{f}^{1}:E^{10}\oplus E^{12} \xrightarrow{\cong} E^{11}$,
as the composition
$$E^{10}\oplus E^{12} \xrightarrow{(f^{10})^{-1}\oplus (f^{12})^{-1}}
E^{00}
  \oplus E^{02}\oplus E^{20} \oplus E^{22}  \xrightarrow{f^{11}} E^{11}.$$

Let $(E,f)$ be a split $n$-cube. For every $\bj\in
\{0,1,2\}^n$, we define a morphism
$$\hat{f}^{\bj}: \Sp(\partial_l^1E)^{\bj}\rightarrow
(\partial_l^1E)^{\bj}$$   as the composition of the isomorphisms
$$\xymatrix{\bigoplus_{\bm\in\{0,2\}^{s(\bj)}
}(\partial_l^1E)^{\sigma_{u(\bj)}^{\bm}(\bj)}\ar[r]^{\hat{f}^{\bj}}
\ar[d]^{\cong}_{\oplus
(\partial_l^1f)^{-1}} & (\partial_l^1E)^{\bj}  \\
\bigoplus_{\bm\in\{0,2\}^{s(\bj)}} (\partial_l^0E \oplus
\partial_l^2E)^{\sigma_{u(\bj)}^{\bm}(\bj)} \ar[r]_{\cong} &
\bigoplus_{\bm\in\{0,2\}^{s(\bj)+1}} E^{\sigma_{u(s_l^1(\bj))}^{\bm}(s_l^1(\bj))
\ar[u]^{\cong}_{f^{\bj}}}, }$$ where the bottom arrow is the canonical
isomorphism. Then, we define
$$\partial_l^1(E,f):=(\partial_l^1E, \hat{f}). $$
With this definition of $\partial_l^1$, the commutation rule
\eqref{commutfaces} is satisfied.  Therefore, we have
proved the following proposition.
\begin{prop}
The morphism
 $$d=\sum_{l=1}^n\sum_{i=0,1,2}(-1)^{i+l}\partial_l^i: \Z \Sp\nolimits_n(X) \rightarrow \Z \Sp\nolimits_{n-1}(X)$$
 makes $\Z \Sp\nolimits_*(X)$ into a chain complex. Moreover, the morphism
$\Z \Sp\nolimits_*(X) \rightarrow \Z C_*(X)$ obtained by forgetting the splittings is
a chain morphism.
\end{prop}
\vist

\begin{obs} Observe that due to \eqref{partial1}, the morphism
$$\Sp: \Z C_*(X) \rightarrow \Z \Sp\nolimits_*(X)$$ is not a chain morphism.
\end{obs}

\subsection{An intermediate complex for the Adams operations}
Here we introduce the chain
complex that  serves as the target for the Adams operations defined on the chain complex of split
cubes. We then construct a morphism from this new chain complex to  the
original chain complex of cubes $\Z C_*(X)$.

Let $k\geq 1$. For every $n\geq 0$ and $i=1,\dots,k-1$, we define
\begin{eqnarray*}
G^k_1(X)_n & := & IC_{1}^{k}(C_n(X)) \\ & := & \{\textrm{acyclic cochain
complexes of length }k\textrm{ of  }n-\textrm{cubes} \}, \\
G^{i,k}_2(X)_n &:= & IC_{2}^{k-i,i}(C_{n}(X)) \\  &:=&
\{\textrm{2-iterated acyclic cochain complexes of lengths }(k-i,i) \\ && \
\textrm{ of }n-\textrm{cubes}
 \}.
\end{eqnarray*}  The differential of $\Z C_*(X)$ induces a
differential on the graded abelian groups
$$\Z G^{i,k}_2(X)_*:=\bigoplus_n \Z G^{i,k}_2(X)_n\qquad  \textrm{and}\qquad \Z G^k_1(X)_*:=\bigoplus_n
\Z G^k_1(X)_n.$$  That is, if $B\in G^{i,k}_2(X)_n$, then for every
$r,s$, $B^{r,s}$ is an $n$-cube. Define $\partial_l^i(B)$ to be
the $2$-iterated cochain complex of lengths $(k-i,i)$ of
$(n-1)$-cubes given by
$$\partial_l^i(B)^{r,s}:=\partial_l^i(B^{r,s}) \in C_{n-1}(X),\qquad\textrm{ for every }r,s.$$
Then the differential of $B$ is defined as $$d(B)=\sum_{i=1}^n
\sum_{l=0}^2 (-1)^{i+l}\partial_l^i(B).$$ If $A\in G^k_1(X)_n$,
then for every $r$, $A^{r}$ is an $n$-cube and the differential is
defined analogously.

For every $n$, the simple complex associated to a $2$-iterated cochain complex induces a morphism
$$\Phi^i:\Z G^{i,k}_2(X)_n\rightarrow \Z G^k_1(X)_n. $$  That is, for every $B\in G^{i,k}_2(X)_n$, $\Phi^i(B)$ is the
exact sequence of $n$-cubes
$$\Phi^i(B):= 0\rightarrow B^{00} \rightarrow \dots \rightarrow \bigoplus_{j_1+j_2=j}
B^{j_1,j_2} \rightarrow \dots \rightarrow B^{k-i,i} \rightarrow 0$$ with
morphisms given by
\begin{eqnarray*}
 B^{j_1,j_2} &\rightarrow &  B^{j_1+1,j_2}\oplus B^{j_1,j_2+1}
\\ b &\mapsto & d^1(b)+(-1)^{j_1}d^2(b).
 \end{eqnarray*}
One can easily check  that, for every $i=1,\dots,k-1$,
 $\Phi^i$  is a chain morphism.

We define a new chain complex by setting
$$\Z G^{k}(X)_n := \bigoplus_{i=1}^{k-1}\Z G^{i,k}_2(X)_{n-1}  \oplus \Z G^k_1(X)_n. $$
If $B_i\in G^{i,k}_2(X)_{n-1}$, for $i=1,\dots,k-1$, and $A\in G^k_1(X)_n$, the differential is given by
$$d_s(B_1,\dots,B_{k-1},A):= (-dB_1,\dots,-dB_{k-1}, \sum_{i=1}^{k-1}(-1)^{i}\Phi^i(B_i)+dA). $$
Since, for all $i$, the morphisms $\Phi^i$ are chain
morphisms, $d^2=0$ and therefore $(\Z G^k(X)_*=\bigoplus_n \Z G^k(X)_n,d_s)$ is a
chain complex.

Our purpose is to define a chain
morphism from the chain complex $\Z G^k(X)_*$ to the complex of cubes $\Z C_*(X)$. It is constructed
in two steps. First, we define a chain morphism from  $\Z G^k(X)_*$ to the complex
$$\Z C^{arb}_*(X):=\Z IC^{\cdot,2,\dots,2}_*(X).$$
This complex is the chain
complex that in degree $n$ consists of $n$-iterated cochain
complexes of length 2 in directions $2,\dots,n$ and arbitrary
finite length in direction 1. Alternatively, it can be thought of as
the complex of exact sequences of arbitrary finite length of
$(n-1)$-cubes.

Then, we   construct a morphism from $\Z C^{arb}_*(X)$ to $\Z
C_*(X)$, using the splitting of an acyclic cochain complex into short exact sequences.

Let
$$A : 0\rightarrow A^0\rightarrow \dots \rightarrow A^k \rightarrow 0\in G^k_1(X)_n$$
be an acyclic cochain complex of $n$-cubes. We define $\varphi_1(A)$ to be the
``secondary Euler characteristic class'', i.e.
$$\varphi_1(A) = \sum_{p\geq 0}(-1)^{k-p+1} (k-p) A^p \in \Z C_n(X).$$

We choose the signs of this definition in order to agree with
Grayson's definition of Adams operations for $n=0$ in
\cite{Grayson1}. Note that in loc. cit., an exact sequence is
viewed as a chain complex, while here it is viewed as a cochain
complex.

Recall that if $B_i \in G_2^{i,k}(X)_n$, then $B_i$ is a 2-iterated acyclic cochain complex
where $B_i^{j_1j_2}$ is an $n$-cube, for every $j_1,j_2$. We attach to $B_i$ a sum
of exact sequences of $n$-cubes as follows.
\begin{eqnarray*}
\varphi_2(B_i) & = & \sum_{j\geq
0}(-1)^{k-j+1}((k-i-j)B_i^{*,j}+(i-j)B_i^{j,*})
\\ &&+ \sum_{s\geq 1}(-1)^{k-s} (k-s) \sum_{j\geq 0} (B_i^{s-j,j}\rightarrow
\bigoplus_{j'\geq j} B_i^{s-j',j'}\rightarrow
\bigoplus_{j'>j}B_i^{s-j',j'}).\end{eqnarray*} Roughly speaking, the first
summand corresponds to the secondary Euler characteristic of the rows and the
columns. The second summand appears as a correction factor for the fact that
direct sums are not sums in $\Z C_n(X)$.

For every $n$, we define a morphism
\begin{eqnarray}
\Z G^k(X)_{n} & \xrightarrow{\varphi} & \Z C_n^{arb}(X) \label{morphismvarphi} \\
(B_1,\dots,B_{k-1},A) &\mapsto &
\varphi_1(A)+\sum_{i=1}^{k-1}(-1)^{i+1}\varphi_2(B_i).\nonumber
\end{eqnarray}

\begin{lema}
The morphism $\varphi$ is a chain morphism between $\Z G^k(X)_{*}$ and $\Z C_*^{arb}(X)$.
\end{lema}
\begin{proof}
The lemma follows from the two equalities
\begin{eqnarray*}
d\varphi_1(A)&=& \varphi_1(dA), \\
d\varphi_2(B_i) & = & -\varphi_2(dB_i) - \varphi_1(\Phi^i(B_i)),\quad \forall
\ i.
\end{eqnarray*}
The first equality holds as a direct consequence of the definition
of $\varphi_1$. By the definition of the differential of $\Z
G_2^{i,k}(X)_*$,
$$-\varphi_2(dB_i)=\sum_{l=2}^n\sum_{j=0}^2 (-1)^{l+j} \partial_l^j
\varphi_2(B_i).$$ Therefore, it remains to see that
$$\sum_{r\geq 0}(-1)^r \partial_1^r \varphi_2(B_i)= \varphi_1(\Phi^i(B_i)).$$
In other terms, writing
\begin{eqnarray*}
(1) &:=& \sum_{s\geq 1}(-1)^{k-s} (k-s) \sum_{j\geq 0}
\sum_{r=0}^2(-1)^r\partial_1^{r}\big[B_i^{s-j,j}\rightarrow \bigoplus_{j'\geq
j} B_i^{s-j',j'}\rightarrow \bigoplus_{j'>j}B_i^{s-j',j'}\big], \\
(2) &:=& \sum_{r\geq 0}(-1)^r\partial_1^{r}\Big[\sum_{j\geq
0}(-1)^{k-j+1}((k-i-j)B_i^{*,j}+(i-j)B_i^{j,*})\Big],
\end{eqnarray*}
we want to see that
\begin{equation}\label{eq1}
(1)+(2) = \varphi_1(\Phi^i(B_i)) = \sum_{s\geq 0}(-1)
^{k-s+1}(k-s)\bigoplus_{j\geq 0}B_i^{s-j,j}.
\end{equation}
By a telescopic argument, the first term is
\begin{eqnarray*}
(1) &=&\Big[\sum_{s\geq 1}(-1)^{k-s} (k-s)\sum_{j\geq 0}B_i^{s-j,j}\Big]-\Big[
\sum_{s\geq 1}(-1)^{k-s} (k-s) \bigoplus_{j\geq 0} B_i^{s-j,j}\Big]\\
&=&\sum_{s\geq 1}(-1)^{k-s} (k-s)\sum_{j\geq 0}B_i^{s-j,j}
+\varphi_1(\Phi^i(B_i))-(-1)^{k-1}k B_i^{00} \\
&=& \sum_{s\geq 0}(-1)^{k-s} (k-s)\sum_{j\geq 0}B_i^{s-j,j}
+\varphi_1(\Phi^i(B_i)).
\end{eqnarray*}
The second term is
\begin{eqnarray*}
(2) &=& \Big[\sum_{r\geq 0}(-1)^r \sum_{j\geq 0}(-1)^{k-j+1} (k-i-j)
B_i^{r,j}\Big]\\ && +\Big[\sum_{r\geq 0}(-1)^r \sum_{j\geq 0} (-1)^{k-j+1}(i-j)B_i^{j,r}\Big] \\ &=&
\sum_{s\geq 0}(-1)^{k-s+1}\sum_{j\geq 0} (k-i-j) B_i^{s-j,j}+
\sum_{s\geq 0}(-1)^{k-s+1}\sum_{r\geq 0} (i-s+r) B_i^{s-r,r} \\
&=&\sum_{s\geq 0}(-1)^{k-s+1} (k-s)  \sum_{j\geq 0}  B_i^{s-j,j}.
\end{eqnarray*}
Adding the two expressions, \eqref{eq1} is proved.
\end{proof}

The final step is the construction of a morphism from $\Z
C^{arb}_*(X)$ to $\Z C_*(X)$. Recall that, by definition, an
element of $\Z C_m^{arb}(X)$ is a finite length exact sequence of
$(m-1)$-cubes. The idea is to break this exact sequence into short
exact sequences, obtaining a collection of short exact
sequences of $(m-1)$-cubes, hence, a collection of $m$-cubes (see
remark \ref{cubes22}).

Let
$$0 \rightarrow A^0 \xrightarrow{f^0} \cdots \xrightarrow{f^{j-1}} A^j  \xrightarrow{f^j}
\cdots \xrightarrow{f^{r-1}}
 A^r \rightarrow 0$$
 be an exact sequence of $(m-1)$-cubes i.e. an element of  $C_m^{arb}(X)$.
 Let $\mu^j(A)$ be the short exact sequence of $(m-1)$-cubes defined by
\begin{eqnarray*}
\mu^j(A) &:& 0\rightarrow \ker f^j \rightarrow A^j \rightarrow \ker
f^{j+1}\rightarrow 0,\quad j=0,\dots,r-1.
\end{eqnarray*}
It is the $m$-cube that along the first direction is given by:
$$\partial_1^0 (\mu^j(A))=\ker f^j, \quad \partial_1^1(\mu^j(A))=A^j, \quad \textrm{and}\quad
\partial_1^2(\mu^j(A))=\ker f^{j+1}.$$
We define $\mu$ by
$$ \Z C_m^{arb}(X)  \xrightarrow{\mu}  \Z C_m(X) \qquad  A  \mapsto  \sum_{j\geq 0} (-1)^{j-1}\mu^j(A).
$$

The next lemma follows from a direct computation.
\begin{lema}\label{mus}
The map $\mu$ is a chain morphism.
\end{lema}
 \vist

\subsection{Ideas of the definition of the Adams operations on split
cubes}\label{adams2} Let $X$ be a scheme.
In order to enlighten the forthcoming construction, in this section we give examples and the outline of the definition
of Adams operations on split cubes. The starting point is
the use of the Koszul complex, as considered by Grayson in \cite{Grayson1}.

Let $\Z SG^k(X)_*$  be the chain complex obtained like $\Z
G^k(X)_*$ by considering split cubes. That is, considering the
groups
\begin{eqnarray*}
SG^k_1(X)_n & := & IC_{1}^{k}(\Sp\nolimits_n(X)), \\
SG^{i,k}_2(X)_n &:= & IC_{2}^{k-i,i}(\Sp\nolimits_{n}(X)).
\end{eqnarray*}
Observe that there is a natural morphism
$$\Z SG^k(X)_* \rightarrow \Z G^k(X)_*$$
obtained by forgetting the splitting.

For every $k\geq 1$, we construct a morphism,
$$\Z \Sp\nolimits_n(X) \xrightarrow{\Psi^k} \Z SG^k(X)_n,$$
which composed with $\mu\circ \varphi$,  gives a morphism
$$\Z \Sp\nolimits_n(X) \xrightarrow{\Psi^k} \Z C_n(X).$$

\begin{df} Let $E\in \Sp\nolimits_0(X)$ and $k\geq 1$. We define the element $\Psi^k(E) \in SG^k(X)_0=SG^k_1(X)_0$
to be the \emph{$k$-th Koszul complex of $E$}, i.e. the exact
sequence
$$0\rightarrow \Psi^k(E)^0 \rightarrow \dots\rightarrow \Psi^k(E)^k \rightarrow 0$$
with
$$\Psi^k(E)^p= E{\cdot}\stackrel{p}{\dots}{\cdot} E \otimes E\wedge \stackrel{k-p}{\dots}
\wedge E= S^pE \otimes \bigwedge\nolimits^{k-p}E.$$
\end{df}
Observe that, for $k=1$, we have
$$\Psi^1(E): 0 \rightarrow E \xrightarrow{=} E \rightarrow 0.$$

By definition, the Koszul complex is functorial. Moreover, it has a very good
behavior with direct sums.
\begin{lema} If $E$ and $F$ are two locally free sheaves on $X$, then there is a canonical
isomorphism  of exact sequences
\begin{equation}\label{direct}
\Psi^k(E\oplus F)\cong \bigoplus_{m=0}^k \Psi^{k-m}(E)\otimes
\Psi^m(F),\qquad \forall \ k.\end{equation}
\end{lema}
\vist This identification plays a key role in the construction of
the Adams operations.

The definition of $\Psi^k(E)$ of a general split $n$-cube $E$
is given by a combinatorial formula on the Adams operations
$\Psi^k(E^{\bj})$, $\bj\in \{0,1,2\}^n$, of the locally free
sheaves in the cube. In order to understand how the combinatorial
formula of the upcoming definition \ref{adams1} arises,  we explain here
the low degree cases. We give the detailed construction of the
Adams operations for $n=1$, with $k=2,3$, and for $n=2$, $k=2$. We
extract from these examples the key facts that enable us to set
the general formula.

\medskip
{\bf Adams operations in the case $\mathbf{n=1}$, $\mathbf{k=2}$.}
Let $n=1$ and $k=2$, and let
 $E=[E^0,E^2]$. Recall that this notation means that $E$ is the $1$-cube $E^0\rightarrow E^0\oplus E^2 \rightarrow E^2$.
 Our aim is to define $\Psi^2(E)$ in such a way that its
 differential is exactly $$-\Psi^2(E^0)+\Psi^2(E^0\oplus E^2)-\Psi^2(E^2).$$
Consider the two exact sequences
\begin{eqnarray*}
C_0(E)&:=&[\Psi^2(E^0), \Psi^1(E^0)\otimes \Psi^1(E^2)]\in G_1^2(X)_1,\\
C_1(E)&:=&[\Psi^2(E^0)\oplus \Psi^1(E^0)\otimes \Psi^1(E^2),  \Psi^2(E^2)]\in
G_1^2(X)_1.
\end{eqnarray*}
Then,
\begin{eqnarray*}
d(C_0(E)+C_1(E))&=& -\Psi^2(E^0)-\Psi^1(E^0)\otimes \Psi^1(E^2) \\ &&
+\Psi^2(E^0)\oplus \Psi^1(E^0)\otimes \Psi^1(E^2)\oplus\Psi^2(E^2)-\Psi^2(E^2).
\end{eqnarray*}
Observe now that by the isomorphism \eqref{direct},
$$\Psi^2(E^0)\oplus \Psi^1(E^0)\otimes \Psi^1(E^2)\oplus\Psi^2(E^2) \cong \Psi^2(E^0\oplus E^2). $$
 We define then $\widetilde{C}_1(E)$ to be the exact
sequence $C_1(E)$ modified by means of  this isomorphism, that is
$$\widetilde{C}_1(E):\ \Psi^2(E^0)\oplus \Psi^1(E^0)\otimes \Psi^1(E^2) \rightarrow \Psi^2(E^0\oplus E^2)
\rightarrow  \Psi^2(E^2).$$

Finally, observe that the extra term $\Psi^1(E^0)\otimes
\Psi^1(E^2)$ is the simple complex associated to the $2$-iterated
complex of length $(1,1)$
$$ \xymatrix{
E^0\otimes E^0 \ar[r] \ar[d] & \ar[d]E^0\otimes E^2 \\
 E^2\otimes E^0 \ar[r] & E^2\otimes E^2
 } $$
Hence, viewed as a $2$-iterated complex, $\Psi^1(E^0)\otimes \Psi^1(E^2)\in G^1_{2,2}(X)_0$.
We conclude that the differential of
$$\Psi^2(E):=(\Psi^1(E^0)\otimes
\Psi^1(E^2), C_0(E)+\widetilde{C}_1(E))\in \Z G_2^{1,2}(X)_0 \oplus \Z G_1^2(X)_1
$$ is exactly $-\Psi^2(E^0)+\Psi^2(E^0\oplus E^2)-\Psi^2(E^2)$ as desired.

For an arbitrary split $1$-cube $(E,f)$, we define:
\begin{enumerate*}[$\blacktriangleright$]
\item $ \widetilde{C}_0(E,f):= C_0(\Sp(E)) \in G_1^2(X)_1. $
\item $\widetilde{C}_1(E,f)$ is the exact sequence  obtained changing, via the given
splitting $f:E^1\cong E^0\oplus E^2$, the terms $\Psi^2(E^0\oplus E^2)$ in $\widetilde{C}_1(\Sp(E))$ by
$\Psi^2(E^1)$:
$$\xymatrix{  & \Psi^2(E^0\oplus E^2) \ar[dr]\ar[d]^{\cong} & \\
  \Psi^2(E^0)\oplus \Psi^1(E^0)\otimes \Psi^1(E^2)  \ar[r] \ar[ur] &  \Psi^2(E^1) \ar[r] &
  \Psi^2(E^2).}
$$
\end{enumerate*}
We define then
$$\Psi^2(E,f):=(\Psi^1(E^0)\otimes
\Psi^1(E^2), \widetilde{C}_0(E,f)+\widetilde{C}_1(E,f))\in \Z G_2^{1,2}(X)_0 \oplus \Z G_1^2(X)_1.
$$

\textbf{  Adams operations in the case $\mathbf{n=1}$, $\mathbf{k=3}$.}
Let $E=[E^0,E^2]$ as above. Our aim now is to define $\Psi^3(E)$
in such a way that its differential is
$$-\Psi^3(E^0)+\Psi^3(E^0\oplus E^2)-\Psi^3(E^2).$$
We consider the exact sequences,
\begin{eqnarray*}
C_0(E)&:=&[\Psi^3(E^0), \Psi^2(E^0)\otimes \Psi^1(E^2)],\\
C_1(E)&:=&[\Psi^3(E^0)\oplus \Psi^2(E^0)\otimes \Psi^1(E^2),
  \Psi^1(E^0)\otimes \Psi^2(E^2)], \\
  C_2(E)&:=&[\Psi^3(E^0)\oplus \Psi^2(E^0)\otimes \Psi^1(E^2)\oplus
  \Psi^1(E^0)\otimes \Psi^2(E^2),\Psi^3(E^2)].
  \end{eqnarray*}
Then, we define $\widetilde{C}_2(E)$ to be the exact sequence obtained from $C_2(E)$ by
exchanging $$\Psi^3(E^0)\oplus
\Psi^2(E^0)\otimes \Psi^1(E^2)\oplus
\Psi^1(E^0)\otimes \Psi^2(E^2)\oplus\Psi^3(E^2)$$ with $\Psi^3(E^0\oplus
E^2)$ by the isomorphism  \eqref{direct}. That is, $\widetilde{C}_2(E)$ is the exact sequence
$$ \Psi^3(E^0)\oplus \Psi^2(E^0)\otimes \Psi^1(E^2)\oplus
  \Psi^1(E^0)\otimes \Psi^2(E^2) \rightarrow \Psi^3(E^0\oplus
E^2) \rightarrow \Psi^3(E^2).$$
Then, the differential of $C_0(E)+C_1(E)-\widetilde{C}_2(E)$ is
$$-\Psi^3(E^0) +\Psi^3(E^0\oplus
E^2) - \Psi^3(E^2) - \Psi^2(E^0)\otimes \Psi^1(E^2)
 - \Psi^1(E^0)\otimes \Psi^2(E^2). $$
As in the previous example, we have $ \Psi^2(E^0)\otimes
\Psi^1(E^2)\in \Z G_2^{1,3}(X)_0$ and $\Psi^1(E^0)\otimes
\Psi^2(E^2) \in \Z G_2^{2,3}(X)_0$. Hence, the differential of
$$\Psi^3(E):=( \Psi^2(E^0)\otimes \Psi^1(E^2),\Psi^1(E^0)\otimes \Psi^2(E^2),
 C_0(E)+C_1(E)+\widetilde{C}_2(E))$$
 is exactly $-\Psi^3(E^0)+\Psi^3(E^0\oplus E^2)-\Psi^3(E^2)$ as desired.

Finally, for an arbitrary split $1$-cube $(E,f)$,
$$\widetilde{C}_0(E,f):=C_0(\Sp(E)),\qquad
\widetilde{C}_1(E,f):=C_1(\Sp(E)),$$ and $\widetilde{C}_2(E,f)$ is
defined by changing the term $\Psi^3(E^0\oplus E^2)$ in $
\widetilde{C}_2(\Sp(E))$ by $\Psi^3(E^1)$ by means of the
isomorphism induced by $f$.

\textbf{Adams operations in the case $\mathbf{n=2}$, $\mathbf{k=2}$.}
Let $E=\left[\begin{array}{cc} E^{00} & E^{02} \\
E^{20} & E^{22} \end{array} \right]$. Then, we define the
following terms of $\Z G_1^{2}(X)_2$: {\small $$
\setlength\arraycolsep{2pt}\begin{array}{rcl}
 C_{00}(E) &:=& \left[  \begin{array}{cc} \Psi^2(E^{00})\  &\  \Psi^1(E^{00})\otimes
  \Psi^1(E^{02}) \vspace{0.2cm}
\\ \Psi^1(E^{00})\otimes \Psi^1(E^{20}) \
 &\ \Psi^1(E^{00})\otimes \Psi^1(E^{22})\oplus \Psi^1(E^{02})  \otimes \Psi^1(E^{20})
  \end{array} \right], \\ \\
    C_{10}(E) &:=& \left[\begin{array}{cc} \begin{array}{c}\Psi^2(E^{00})\\ \oplus \Psi^1(E^{00})\otimes
   \Psi^1(E^{20}) \end{array} \  & \begin{array}{c}\Psi^1(E^{00})\otimes \Psi^1(E^{02})\oplus \Psi^1(E^{00})\otimes
   \Psi^1(E^{22}) \\ \oplus
  \Psi^1(E^{02})\otimes\Psi^1(E^{20}) \end{array} \    \\ \\ \ \Psi^2(E^{20})
  &\  \Psi^1(E^{20})\otimes \Psi^1(E^{22})\end{array}
   \right], \\ \\
     C_{01}(E)&:=& \left[\begin{array}{cc} \Psi^2(E^{00})\oplus \Psi^1(E^{00})\otimes
   \Psi^1(E^{02}) &
\Psi^2(E^{02}) \vspace{0.2cm} \\
 \begin{array}{c} \Psi^1(E^{00})\otimes
\Psi^1(E^{20})  \oplus \Psi^1(E^{00})\otimes \Psi^1(E^{22})\\
\oplus \Psi^1(E^{02})\otimes\Psi^1(E^{20})
\end{array}\ & \Psi^1(E^{02})\otimes \Psi^1(E^{22})\end{array} \right],\\ \\
C_{11}(E) &:=&  \left[\begin{array}{cc}
\begin{array}{c}\Psi^2(E^{00})\oplus \\
 \Psi^1(E^{00})\otimes \Psi^1(E^{02})\oplus
 \Psi^1(E^{00})\otimes \Psi^1(E^{20}) \\ \oplus  \Psi^1(E^{00})\otimes
\Psi^1(E^{22}) \oplus
\Psi^1(E^{02})\otimes\Psi^1(E^{20})\end{array}\  & \begin{array}{c} \Psi^2(E^{02})
\\ \oplus \Psi^1(E^{02})\otimes
\Psi^1(E^{22}) \end{array}\vspace{0.2cm} \\
\Psi^2(E^{20})  \oplus \Psi^1(E^{20})\otimes \Psi^1(E^{22})  &
\Psi^2(E^{22})
\end{array}\right].\end{array}$$ }

The faces of each of these cubes are as follows (up to the isomorphism \eqref{direct}):
\begin{enumerate*}[$\rhd$]
\item Terms that are summands of $\Psi^2(\partial_i^jE)$:
$$
\begin{array}{rclcrcl }
\partial_1^0C_{00}(E)& = & C_0(\partial_1^0E), & \qquad & \partial_2^0C_{00}(E)& =& C_0(\partial_2^0E), \\
\partial_1^0C_{01}(E)& =& C_1(\partial_1^0E), &  \qquad &\partial_2^0C_{10}(E)& = & C_1(\partial_2^0E), \\
\partial_1^1 C_{10}(E)& = & C_0(\partial_1^1E), &  \qquad &\partial_2^1C_{01}(E)& =& C_0(\partial_2^1E), \\
\partial_1^1 C_{11}(E)& = & C_1(\partial_1^1E), &  \qquad &\partial_2^1C_{11}(E)& =& C_1(\partial_1^1E), \\
\partial_1^2 C_{10}(E)& = & C_0(\partial_1^2E), & \qquad & \partial_2^2C_{01}(E)& =& C_0(\partial_2^2E), \\
\partial_1^2 C_{11}(E)& = & C_1(\partial_1^2E), & \qquad & \partial_2^2C_{11}(E)& =& C_1(\partial_2^2E).
\end{array}$$
\item Terms that are a direct sum of a tensor product of complexes: $$\partial_1^2C_{00}(E), \quad \partial_2^2C_{00}(E),
\quad \partial_2^2C_{10}(E),\quad \partial_1^2C_{01}(E).$$
\item Terms that cancel each other:
$$
\begin{array}{rclcrcl}
\partial_1^1C_{00}(E)& = & \partial_1^0 C_{10}(E), & \qquad & \partial_2^1C_{00}(E)& = & \partial_2^0 C_{01}(E), \\
\partial_2^1C_{10}(E)& = & \partial_2^0 C_{11}(E), & \qquad & \partial_1^1C_{01}(E)& = & \partial_1^0 C_{11}(E).
\end{array}$$
\end{enumerate*}
It follows that the  differential of
$C_{00}(E)+C_{10}(E)+C_{01}(E)+C_{11}(E)$ is $\Psi^2(dE)$ plus
some terms which are  a direct sum of a tensor product of complexes.
These tensor product complexes can be viewed as $2$-iterated
cochain complexes of lengths $(1,1)$ of exact sequences. These
terms are added in $\Z G_{2}^{1,2}(X)_1$.

Finally, for every split $2$-cube $(E,f)$, $\Psi^2(E,f)$ is defined by
modifying the appropriate locally free sheaves in each $C_i(E)$ by
means of the splitting $f$.

\medskip
\textbf{Outline of the definition of $\Psi^k$.}
The given examples suggest that the general procedure can be as follows:
\begin{enumerate*}[$\rhd$]
\item First, for every split $n$-cube $(E,f)$, the direct sum $n$-cubes $C_i(E)$ are defined by a purely
combinatorial formula on the Adams operations of the locally free
sheaves $E^{\bj}$, $\bj\in
\{0,2\}^n$. \item The previous construction is modified
by the isomorphism \eqref{direct}. \item The entries of $C_i(E)$ which give the terms
$C_i(\partial_l^1E)$ in the differential, are modified by the
morphisms induced by the splitting $f$.
\end{enumerate*}

From the examples, the key ideas that lead to the general
combinatorial formula of $C_i(E)$ can also be extracted:
\begin{enumerate*}[$\rhd$]
\item At each step,
some entries in $\partial_r^0$ are constructed by taking the direct sum
$\partial_r^0\oplus
\partial_r^2$ in a previous cube (where ``previous'' refers to the order $\leq$ for the subindices in $C_*(E)$).
\item The new entries (not being direct sums of previous cubes) are direct sums of
summands of the form $\Psi^{k_1}(E^{2\bn_1})\otimes \cdots \otimes
\Psi^{k_r}(E^{2\bn_r})$ satisfying:
\begin{enumerate}[a)]
\item $\sum_s k_s= k$. \item In the position $2\bj$  of the cube
$C_{\bi}(E)$, $\sum_s k_s\bn_s=\bj+\bi$. \item Observe that in the
example $n=2$, all the entries in $C_{00}$ are new, and the new
entries for $C_{10}$ are in the positions $(2,0),(2,2)$ and for
$C_{11}$ in $(2,2)$. Hence the new entries will correspond to the
multi-indexes $\bj$ such that $\bj\geq \nu(\bi)$ (recall that
$\nu(\bi)$ is the characteristic of the multi-index $\bi$).
\end{enumerate}

\end{enumerate*}

\subsection{Definition of the cubes $C_i(E)$}

Let $(E,f)\in \Sp\nolimits_n(X)$ and fix $k\geq 1$. For every $\bi
\in [0,k-1]^n$, we define an exact sequence of direct sum cubes
$C_{\bi}(E)\in \Z SG^k_1(X)_n$. This definition is purely
combinatorial and does not depend on the splitting $f$.

Let
$$L_k^r=\{\bk=(k_1,\dots,k_r)\ |\  |\bk|=k\textrm{ and }k_s\geq 1, \forall
s\}$$ be the set of partitions of length $r$ of $k$.  Then, for
every integer $n\geq 0$ and every multi-index $\bm$ of length $n$,
we define a new set of indices by:

$$\Lambda_k^n(\bm)=\bigcup_{r\geq 1}
\Big\{(\bk,\bn^1,\dots,\bn^r)\in L_k^r\times (\{0,1\}^n)^r\ \Big|
\sum k_s \bn^s = \bm, \ \bn^1\prec\dots\prec\bn^r\Big\}.$$

\begin{df}\label{adams1}
Let $(E,f)\in \Sp_n(X)$. For every $\bi \in [0,k-1]^n$, let
$C_{\bi}(E)\in SG^k_1(X)_n$ be the exact sequence of direct sum $n$-cubes, such that, for every $\bj\in \{0,1\}^n$, the position
$2\bj$ is given as follows:
\begin{enumerate}[(i)]
\item If $\bj \geq \nu(\bi)$, then
\begin{equation}\label{c1}
C_{\bi}(E)^{2\bj}= \bigoplus_{\Lambda_k^n(\bj+\bi)}
\Psi^{k_1}(E^{2\bn^1})\otimes \cdots \otimes
\Psi^{k_r}(E^{2\bn^r}).\end{equation} \item If $\bj \ngeq \nu(\bi)$, then
\begin{equation}\label{c2} C_{\bi}(E)^{2\bj}= \bigoplus_{\bj \leq \bm \leq \nu(\bi)\cup \bj}
C_{\bi-\nu(\bi)\cdot \bj^c}(E)^{2\bm}.\end{equation}
\end{enumerate}
\end{df}
In order to simplify the notation, we will denote by $r$ the
length of $\bk\in \Lambda_k^n(\bj+\bi)$ in the future occurrences
of the sum \eqref{c1}. Observe that the definition of $C_i(E)$ for a split cube
$(E,f)$ does not depend on $f$.

\begin{obs}
First of all, observe that in equation \eqref{c2},
$\bi-\nu(\bi)\cdot \bj^c\in [0,k-1]^n$, i.e. for every $s$, $0\leq
(\bi-\nu(\bi)\cdot \bj^c)_s$:
\begin{enumerate*}[$\rhd$]
\item If $i_s=0$, then $\nu(\bi)_s=0$ and hence
$(\bi-\nu(\bi)\cdot \bj^c)_s=0$. \item If $i_s>0$, then
$\nu(\bi)_s=1$ and since $(\bj^c)_s=0,1$, we have $i_s\geq
\nu(\bi)_s{\cdot}(\bj^c)_s$.
\end{enumerate*}
\end{obs}

\begin{obs}\label{inductioni} Observe that equations \eqref{c1} and \eqref{c2} define
$C_{\bi}(E)^{2\bj}$ for all $\bj\in \{0,1\}^n$. This follows from
the following facts:
\begin{enumerate*}[$\rhd$]
\item  Since $\nu(\bj) \geq (0,\dots,0)$, equation \eqref{c1}
defines $C_{\bi}(E)$ for $\vline\ \bi\ \vline = 0$. \item If $\bj
\ngeq \nu(\bi)$, then
$$\vline\ \bi-\nu(\bi)\cdot \bj^c\ \vline <\ \vline\ \bi\ \vline\
.$$ Indeed, an equality  would imply that $\nu(\bi)\cdot \bj^c=0$
and hence that for all $r$ such that $i_r \neq 0$, $\bj^c_r=0$,
concluding that $\bj \geq \nu(\bi)$.
\end{enumerate*}
\end{obs}

\begin{obs} Observe that equation \eqref{c2} also holds trivially for $\bj \geq \nu(\bi)$,
because in this case $\nu(\bi)\cdot \bj^c=0$ and $\nu(\bi)\cup \bj=\bj$.  We
will use this observation in some proofs when only combinatorial questions are
involved.
\end{obs}

\begin{obs}
The direct sum of more than two terms means the consecutive direct
sums of two objects under the lexicographic order in the
subindices. In order to prove some equalities, it will be
necessary to reorder the indices, and then return to the original
order. For the sake of simplicity, we will not write the
canonical isomorphisms used at every step and will just write
equalities. The reader should bear this remark in mind throughout
this section.
\end{obs}

\subsection{Faces of the cubes $C_{\bi}(E)$}

In this section, we compute  the faces of the cubes $C_{\bi}(E)$.
We fix $k\geq 1$, $\bi \in [0,k-1]^n$, a split $n$-cube $(E,f)\in
\Sp_n(X)$ and $l\in \{1,\dots,n\}$.

\begin{lema}\label{partial0}
$$\partial_l^0 C_{\bi}(E) = \left\{\begin{array}{ll}
 \partial_l^1 C_{\bi-1_l}(E) & \textrm{if }i_l\neq 0, \\[0.2cm]
 C_{\partial_l(\bi)}(\partial_l^0E) & \textrm{if }i_l= 0.
 \end{array}\right.$$
\end{lema}

\begin{proof}
Assume that $i_l\neq 0$. It is enough to see that
$$\partial_l^0 C_{\bi}(E)^{2\bj}= \partial_l^1C_{\bi-1_l}(E)^{2\bj},\qquad \forall \ \bj\in \{0,1\}^{n-1}.$$
Observe that
\begin{eqnarray*}
\partial_l^0 C_{\bi}(E)^{2\bj}&=&
C_{\bi}(E)^{2s_l^0(\bj)}=\bigoplus_{s_l^0(\bj) \leq \bm \leq \nu(\bi)\cup
s_l^0(\bj)}
C_{\bi-\nu(\bi)\cdot s_l^0(\bj)^c}(E)^{2\bm}, \\
\partial_l^1C_{\bi-1_l}(E)^{2\bj}&=& C_{\bi-1_l}(E)^{2s_l^0(\bj)}\oplus
C_{\bi-1_l}(E)^{2s_l^1(\bj)},
\end{eqnarray*}
with
\begin{eqnarray}
C_{\bi-1_l}(E)^{2s_l^0(\bj)} &=& \bigoplus_{s_l^0(\bj) \leq \bm \leq
\nu(\bi-1_l)\cup s_l^0(\bj)}
C_{\bi-1_l-\nu(\bi-1_l)\cdot s_l^0(\bj)^c}(E)^{2\bm}, \qquad \label{p1}\\
C_{\bi-1_l}(E)^{2s_l^1(\bj)}  &=& \bigoplus_{s_l^1(\bj) \leq \bm \leq
\nu(\bi-1_l)\cup s_l^1(\bj)} C_{\bi-1_l-\nu(\bi-1_l)\cdot
s_l^1(\bj)^c}(E)^{2\bm}. \label{p2}
\end{eqnarray}
Let us compute each term  separately. We start with \eqref{p2}.
Since $s_l^1(\bj)_l=1$, we see that $\bm_l=1$ for all indices
$\bm$ of the direct sum. Moreover, since $s_l^1(\bj)^c_l=0$,
$s_l^0(\bj)^c_l=1$, and $\nu(\bi)_l=1$, we obtain that
$$(\bi-1_l-\nu(\bi-1_l)\cdot s_l^1(\bj)^c)_l=i_l-1_l=(\bi-\nu(\bi)\cdot s_l^0(\bj)^c)_l.$$
Since it is clear that for all $t\neq l$, $(\bi-\nu(\bi)\cdot
s_l^0(\bj)^c)_t =(\bi-1_l-\nu(\bi-1_l)\cdot s_l^1(\bj)^c)_t$, we
see that $$\bi-\nu(\bi)\cdot s_l^0(\bj)^c =
\bi-1_l-\nu(\bi-1_l)\cdot s_l^1(\bj)^c.$$ Thus,
$$C_{\bi-1_l}(E)^{2s_l^1(\bj)}= \bigoplus_{\substack{s_l^0(\bj) \leq \bm \leq \nu(\bi)\cup
s_l^0(\bj), \\ \bm_l=1}} C_{\bi-\nu(\bi)\cdot s_l^0(\bj)^c}(E)^{2\bm}. $$

All that remains is to see that
$$C_{\bi-1_l}(E)^{2s_l^0(\bj)} = \bigoplus_{\substack{s_l^0(\bj) \leq \bm \leq \nu(\bi)\cup
s_l^0(\bj), \\ \bm_l=0}} C_{\bi-\nu(\bi)\cdot s_l^0(\bj)^c}(E)^{2\bm}.$$

We proceed by induction on $i_l$. If $i_l=1$ then  $(\bi-1_l)_l=0$
and hence $(\nu(\bi-1_l)\cup s_l^0(\bj))_l=0$ which means that
$m_l=0$ for all multi-indices $\bm$ in the direct sum \eqref{p1}.
Moreover $(\bi-1_l-\nu(\bi-1_l)\cdot s_l^0(\bj)^c)_l=0$ and
$(\bi-\nu(\bi)\cdot s_l^0(\bj)^c)_l=0$. Therefore,
$$\bi-1_l-\nu(\bi-1_l)\cdot s_l^0(\bj)^c = \bi-\nu(\bi)\cdot s_l^0(\bj)^c$$
and the equality is proven. Let $i_l>1$ and assume that the lemma
is true for $i_l-1$. Then, since $i_l>1$, we have  $\nu(\bi-1_l)=\nu(\bi)$
and $\nu(\bi)_l=1$. Writing $\alpha=\bi-1_l-\nu(\bi-1_l)\cdot
s_l^0(\bj)^c$, we obtain
\begin{align*} C_{\bi-1_l}(E)^{2s_l^0(\bj)} &=
\bigoplus_{\substack{s_l^0(\bj) \leq \bm \leq \nu(\bi-1_l)\cup s_l^0(\bj) \\
m_l=0}} C_{\alpha}(E)^{2\bm} \oplus
\bigoplus_{\substack{s_l^0(\bj) \leq \bm \leq \nu(\bi-1_l)\cup
s_l^0(\bj)\\ m_l=1}}
C_{\alpha}(E)^{2\bm} \\
&=\bigoplus_{\bj \leq \bn \leq \nu(\partial_l(\bi))\cup \bj }
(\partial_l^0 \oplus \partial_l^1)C_{\alpha}(E)^{2\bn}\\ & =
\bigoplus_{\bj \leq \bn \leq \nu(\partial_l(\bi))\cup \bj }
\partial_l^0
C_{\bi-\nu(\bi)\cdot s_l^0(\bj)^c}(E)^{2\bn} \\
&=
\bigoplus_{\substack{s_l^0(\bj) \leq \bm \leq \nu(\bi)\cup s_l^0(\bj), \\
m_l=0}} C_{\bi-\nu(\bi)\cdot s_l^0(\bj)^c}(E)^{2\bm},
\end{align*}
since $(\bi-\nu(\bi)\cdot s_l^0(\bj)^c)_l=i_l-1$ and we can apply
the induction hypothesis in the third equality.

Let us now prove the equality with $i_l=0$. Assume that
$s_l^0(\bj)\geq \nu(\bi)$. Then,
$$\partial_l^0 C_{\bi}(E)^{2\bj}= C_{\bi}(E)^{2s_l^0(\bj)} = \bigoplus_{
\Lambda_k^n(s_l^0(\bj)+\bi)} \Psi^{k_1}(E^{2\bn^1})\otimes \cdots \otimes
\Psi^{k_r}(E^{2\bn^r}) = (*).$$ Since $(s_l^0(\bj)+\bi)_l=0$, $(\sum k_s
\bn^s)_l=0$. Hence, for all $s$, $\bn^s_l=0$ and we obtain
$$(*)=\bigoplus_{
\Lambda_k^{n-1}(\bj+\partial_l(\bi))} \Psi^{k_1}(\partial_l^0E^{2\bn^1})\otimes
\cdots \otimes
\Psi^{k_r}(\partial_l^0E^{2\bn^r})=C_{\partial_l(\bi)}(\partial_l^0E)^{2\bj}.
$$

Finally, if $s_l^0(\bj)\ngeq \nu(\bi)$, the direct sum \eqref{c2}
can be written in the form
$$\partial_l^0 C_{\bi}(E)^{2\bj}=\bigoplus_{\bm,\bn}
C_{\bn}(E)^{2\bm}$$ with $\bm\geq \nu(\bn)$ and $m_l=0$.  We
deduce that $n_l=0$ and hence,
$$\partial_l^0 C_{\bi}(E)^{2\bj}=\bigoplus_{\bm,\bn} C_{\bn}(E)^{2\bm}= \bigoplus_{\bm,\bn}
C_{\partial_l(\bn)}(\partial_l^0E)^{2\bm}=C_{\partial_l(\bi)}(\partial_l^0E)^{2\bj}.
$$ This reduces the proof to the already considered case.
\end{proof}

\begin{lema}\label{partial11} If $\bi_l=k-1$, then \
$\partial_l^2 C_{\bi}(E) = C_{\partial_l(\bi)}(\partial_l^2E). $
\end{lema}
\begin{proof}
Arguing as in the proof of the previous lemma, we limit ourselves to
proving the equality in the case where $s_l^1(\bj)\geq \nu(\bi)$. In
this situation, we obtain
$$(\partial_l^2 C_{\bi}(E))^{2\bj}= C_{\bi}(E)^{2s_l^1(\bj)} = \bigoplus_{
\Lambda_k^n(s_l^1(\bj)+\bi)} \Psi^{k_1}(E^{2\bn^1})\otimes \cdots
\otimes \Psi^{k_r}(E^{2\bn^r}).$$ Since
$(s_l^1(\bj)+\bi)_l=1+k-1=k$, we deduce that for all $s$,
$\bn^s_l=1$ and therefore the lemma is proved.
\end{proof}

The next lemma determines the faces $\partial_l^2$ of $C_{\bi}(E)$
whenever $i_l\neq k-1$.
\begin{lema}\label{tens}
Let $\bj\in \{0,1\}^{n}$ with $j_l=1$ and let $i_l\neq k-1$. Up to
a canonical isomorphism, each of the direct summands of
$C_{\bi}(E)^{2\bj}$ is the tensor product of an exact sequence of
length $(k-i_l-1)$ by an exact sequence of length $(i_l+1)$.
Explicitly, in the equality
$$C_{\bi}(E)^{2\bj}=\bigoplus_{\Lambda_k^n(\bj+\bi)}
 \Psi^{k_1}(E^{2\bn^1})\otimes \cdots \otimes
\Psi^{k_r}(E^{2\bn^r}),$$ the tensor product of the Koszul
complexes corresponding to the multi-indices $\bn$ with $n_l=0$
gives the exact sequence of length $k-i_l-1$, while  the tensor
product of the Koszul complexes corresponding to the multi-indices
$\bn$ with $n_l=1$ gives the exact sequence of length $i_l+1$.
\end{lema}
\begin{proof} If
$\bj\geq \nu(\bi)$, then,
$$C_{\bi}(E)^{2\bj}=\bigoplus_{\Lambda_k^n(\bj+\bi)}
 \Psi^{k_1}(E^{2\bn^1})\otimes \cdots \otimes
\Psi^{k_r}(E^{2\bn^r}).$$ Assume that one of the summands is not a
tensor product. Then there   exists a multi-index $\bn$ with $k\bn
= \bj + \bi$. In particular, $(k\cdot\bn)_l = 1+\bi_l$. But
$(k\cdot \bn)_l$ is either $0$ or $k$, and by hypothesis, $1\leq
1+\bi_l < 1+k-1 = k$, which is a contradiction. If $\bj\ngeq
\nu(\bi)$, then,
$$C_{\bi}(E)^{2\bj}=\bigoplus_{\bj \leq \bm \leq \nu(\bi)\cup \bj}
C_{\bi-\nu(\bi)\cdot \bj^c}(E)^{2\bm}.$$ The condition $\bj\leq
\bm$ implies that $\bm_l=1$. Moreover, $(\bi-\nu(\bi)\cdot
\bj^c)_l=i_l-\bj^c_l \leq i_l < k-1$. This means that every direct
summand is a tensor product of exact sequences. By induction on
$|\bi|$, this is true for any multi-index $\bi$.

Let us prove that every direct summand can be seen as the tensor
product of an exact sequence of length $k-i_l-1$, corresponding to
the multi-indices $\bn$ with $n_l=0$, and one of length $i_l+1$,
corresponding to the multi-indices with $n_l=1$. By an induction
argument, it is enough to prove the result in the case $\bj\geq
\nu(\bi)$. Let $(\bk,\bn^1,\dots,\bn^r)\in \Lambda_k^n(\bj+\bi)$.
Let $s_1,\dots,s_m$ be the indices such that $\bn^{s_j}_{l}=1$ and
let $s'_1,\dots,s'_{r-m}$ be the indices such that
$\bn^{s'_j}_{l}=0$. Since $\sum k_s \bn^s_l=i_l+1$, we see that
$\sum_{j=1}^m k_{s_j}=i_l+1$ and hence
 $$T_1:=\Psi^{k_{s_1}}(E^{2\bn^{s_1}})\otimes \cdots \otimes
\Psi^{k_{s_m}}(E^{2\bn^{s_m}})$$ is an exact sequence of length
$i_l+1$.  Then,
 $$T_0:=\Psi^{k_{s'_1}}(E^{2\bn^{s'_1}})\otimes \cdots \otimes
\Psi^{k_{s'_{r-m}}}(E^{2\bn^{s'_{r-m}}})$$ is an exact sequence of
length $k-i_l-1$ and there is a canonical isomorphism
$$ \Psi^{k_1}(E^{2\bn^1})\otimes \cdots
\otimes \Psi^{k_r}(E^{2\bn^r}) \cong T_0\otimes T_1 $$ as desired.
\end{proof}

\begin{lema}\label{sum}
If $i_l= k-1$, then
$$\partial_l^1
C_{\bi}(E) \cong C_{\partial_l(\bi)}(\partial_l^0E\oplus \partial_l^2E),
$$ with the isomorphism induced by the canonical isomorphism of
the Koszul complex of a direct sum in \eqref{direct}.
\end{lema}
\begin{proof}
The first part of the lemma is lemma \ref{partial0}. Assume then
that $i_l=k-1$. Applying lemmas \ref{partial11} and \ref{partial0}
recursively, we obtain
\begin{equation}\label{c6}
\partial_l^1 C_{\bi}(E)\cong \partial_l^0C_{\bi}(E)\oplus \partial_l^2C_{\bi}(E)=C_{\partial_l(\bi)}(\partial_l^0E)\oplus
 \bigoplus_{a\in [0,k-1]}
\partial_l^2 C_{\bi-a_l}(E). \end{equation}
Then, if $\bj\in \{0,1\}^{n-1}$ satisfies $\bj\geq
\partial_l\nu(\bi)$, we obtain that
$$\begin{array}{l} \partial_l^1
C_{\bi}(E)^{2\bj}=\bigoplus_{\Lambda_k^{n-1}(\bj+\partial_l(\bi))}
\Psi^{k_1}(\partial_l^0E^{2\bn^1})\otimes \cdots \otimes
\Psi^{k_r}(\partial_l^0E^{2\bn^r}) \oplus \vspace{0.3cm}  \\
\hspace{3cm} \bigoplus_{a\in [0,k-1]}\
\bigoplus_{\Lambda_k^{n}(s_l^1(\bj)+\bi-a_l)}
\Psi^{k_1}(E^{2\bn^1})\otimes \cdots \otimes
\Psi^{k_r}(E^{2\bn^r}).
\end{array}$$

On the other hand, by the additivity of $\Psi^k$ in \eqref{direct}, there are
canonical isomorphisms
\begin{eqnarray*}\Psi^{k_1}((\partial_l^0E\oplus \partial_l^2E)^{2\bn^1})
\otimes \cdots \otimes \Psi^{k_r}((\partial_l^0E\oplus
\partial_l^2E)^{2\bn^r})=\phantom{\otimes \cdots \otimes
\Psi^{k_r-m_r}((\partial_l^0E)^{\bn^r})}
\\ \phantom{\Psi^{k_1}(\partial_l^0E\oplus
\partial_l^1E)_{\bn^1} }\cong  \bigoplus_{m_1=0}^{k_1}\cdots\bigoplus_{m_r=0}^{k_r}
\begin{array}{r}
\\ \Psi^{k_1-m_1}(\partial_l^0E^{2\bn^1})\otimes \cdots \otimes
\Psi^{k_r-m_r}(\partial_l^0E^{2\bn^r}) \vspace{0.2cm}\\  \otimes
\Psi^{m_1}(\partial_l^2E^{2\bn^1})\otimes \cdots \otimes
\Psi^{m_r}(\partial_l^2E^{2\bn^r}).\end{array}\end{eqnarray*} Therefore,
$C_{\partial_l(\bi)}(\partial_l^0E\oplus
\partial_l^2E)^{2\bj}$ is canonically isomorphic to
$$ \bigoplus_{\Lambda_k^{n-1}(\bj+\partial_l(\bi))}
\bigoplus_{m_1=0}^{k_1}\cdots\bigoplus_{m_r=0}^{k_r}\  \begin{array}{r}
\\ \Psi^{k_1-m_1}(E^{2s_l^0(\bn^1)})\otimes \cdots \otimes
\Psi^{k_r-m_r}(E^{2s_l^0(\bn^r)}) \vspace{0.2cm}\\  \otimes
\Psi^{m_1}(E^{2s_l^1(\bn^1)})\otimes \cdots \otimes
\Psi^{m_r}(E^{2s_l^1(\bn^r)}).\end{array}
$$

The first summand in \eqref{c6} corresponds to the indices
$m_1,\dots,m_r=0$ in the latter sum. Therefore, we have to see
that the second summand in \eqref{c6} corresponds to the summand
in the latter sum with at least one index $m_i\neq 0$. We will see
that there is a bijection between the sets of multi-indices of
each term.

For every collection
$k_1,\dots,k_r,\bn^1,\dots,\bn^r,m_1,\dots,m_r$ with not all
$m_s=0$, let $a=k-\sum m_s$. Since $\sum m_s\neq 0$, $a\in
[0,k-1]$. Let $s_1,\dots,s_t\in \{1,\dots,r\}$ be the indices for
which $k_{s_l}-m_{s_l}\neq 0$ and let $s'_1,\dots,s'_m\in
\{1,\dots,r\}$ be the indices for which $m_{s'_l}\neq 0$. Then, to
these data correspond the indices $a$, and
\begin{eqnarray*}
\{k'_1,\dots,k_{t+m}'\}&=&
\{k_{s_1}-m_{s_1},\dots,k_{s_t}-m_{s_t},m_{s'_1},\dots,m_{s'_m}\} \\
\hat{\bn}^p& = & \left\{ \begin{array}{ll} s_l^0(\bn^{s_p}) & \textrm{ if
}p=1,\dots,t, \\ s_l^1(\bn^{s_{p-t}})& \textrm{ if }p=t+1,\dots,t+m.
\end{array}\right. \end{eqnarray*}

Conversely, let $a,k_s$ and $\bn^s$ be given. Then, we rearrange
the collection $\bn^s$ by the rule:
$$\bn^1,\dots,\bn^y,\bn^{y+1},\dots,\bn^{x},\bn^{x+1},\dots,\bn^{2x-y},\bn^{2x-y+1},\dots,
\bn^{r}$$ with
$$(\bn^s)_l=\left\{ \begin{array}{ll} 0 & \textrm{if }s=1,\dots,x, \\ 1 &
\textrm{if }s=x+1,\dots,r,
  \end{array}\right. $$
The index $y$ satisfies that for $s=1,\dots,y$ and for $s=2x-y+1,\dots,r$, there is no other
index $s'$ with $\partial_l(\bn^s)=\partial_l(\bn^{s'})$
and for $s=y+1,\dots,x-y$, $\partial_l(\bn^s)=\partial_l(\bn^{s+x-y})$.
Then, the corresponding multi-indices are
\begin{eqnarray*}
(\hat{\bn}^1,\dots,\hat{\bn}^{r-x-y})&=&
(\partial_l(\bn^1),\dots,\partial_l(\bn^x),\partial_l(\bn^{2x-y+1}),\dots,
\partial_l(\bn^{r})),
\\
k'_s&=&\left\{\begin{array}{ll} k_s & \textrm{if }s=1,\dots,y, \\
k_s+k_{s+x-y}&
\textrm{if }s=y+1,\dots,x, \\
k_{s+x-y} & \textrm{if }s=x+1,\dots,r-x-y.
\end{array} \right. \\
m'_s&=&\left\{\begin{array}{ll} 0 & \textrm{if }s=1,\dots,y, \\
k_{s+x-y}&
\textrm{if }s=y+1,\dots,x, \\
k_{s+x-y} & \textrm{if }s=x+1,\dots,r-x-y.
\end{array} \right.
\end{eqnarray*}
The lemma follows from this correspondence.
\end{proof}

\subsection{Definition of the cubes $\wC_{\bi}(E)$}

At this point, we have defined the cubes $C_{\bi}(E)$ for every
split $n$-cube $(E,f)$. Roughly speaking, all that remains is
to change, by means of $f$, the terms corresponding to
$\partial_l^0E\oplus
\partial_l^2E$ by the terms in $\partial_l^1E$.

By the collection of lemmas above, this will be the case whenever
$i_l=k-1$ and $j_l=1$. Therefore, let $\bj\in \{0,1,2\}^n$ and
$\bi\in [0,k-1]^n$. Let
\begin{enumerate*}[$\rhd$]
\item $w(\bj)=(w_1,\dots,w_{r_w(\bj)})$ where $w_1<\dots
<w_{r_w(\bj)}$ are the indices such that $j_{w_m}=1$ and
$i_{w_m}=k-1$. \item $v(\bj)=(v_1,\dots,v_{r_v(\bj)})$ where
$v_1<\dots <v_{r_{v}(\bj)}$ are the indices such that $j_{v_m}=1$
and $i_{v_m}\neq k-1$.
\end{enumerate*}
Then, by lemma \ref{sum},
$$C_{\bi}(E)^{\bj} \cong \bigoplus_{\bm \in \{0,2\}^{r_v(\bj)}} C_{\partial_{w(\bj)}(\bi)}
\Big(\bigoplus_{\bn\in \{0,2\}^{r_w(\bj)}}
E^{\sigma_{w(\bj)}^{\bn}}\Big)^{\partial_{w(\bj)}(\sigma_{v(\bj)}^{\bm}(\bj))}.
$$
Recall that there is an isomorphism
$$\bigoplus_{\bn\in \{0,2\}^r}
E^{\sigma_{w(\bj)}^{\bn}} \stackrel{f}{\cong}
\partial_{w(\bj)}^{\bun}E.$$
This motivates the following definition.

\begin{df}
Let $(E,f)$ be a split $n$-cube and let $\bi\in [0,k-1]^n$. The
$n$-cube $\wC_{\bi}(E)$ is defined by:
\begin{equation}\label{c5}
\wC_{\bi}(E)^{\bj} = \bigoplus_{\bm \in \{0,2\}^{r_w(\bj)}}
C_{\partial_{w(\bj)}(\bi)}
\big(\partial_{w(\bj)}^{\bun}E\big)^{\partial_{w(\bj)}(\sigma_{v(\bj)}^{\bm}(\bj))}.
\end{equation}
The morphisms in $\wC_{\bi}(E)$ are given as follows.
\begin{enumerate}[(i)]
\item If $l$ with $i_l=k-1$ does not exist, then the cube
$\wC_{\bi}(E)$ is split. \item If $i_l=k-1$, then the morphisms in
the cube are induced by the fixed isomorphisms
$\partial_{w(\bj)}^{\bun}(E)\cong \bigoplus_{\bn\in
\{0,2\}^{r_w(\bj)}}
\partial_{w(\bj)}^{\bn} E
$ and the canonical isomorphisms in lemma \ref{sum}.
\end{enumerate}
\end{df}
Since all isomorphisms are fixed, the following proposition is a
consequence of lemmas \ref{partial0}, \ref{partial11}, \ref{tens}
and \ref{sum},
\begin{prop}\label{partials} Let $(E,f)$ be a split $n$-cube, $\bi \in [0,k-1]^n$ and $l\in \{1,\dots,n\}$.
\begin{enumerate}[(i)]
\item If $i_l=0$, then $\partial_l^0 \wC_{\bi}(E)
=\wC_{\partial_l(\bi)}(\partial_l^0E)$. \item If $i_l\neq 0$, then
$\partial_l^0 \wC_{\bi}(E) =\partial_l^1\wC_{\bi-1_l}(E)$. \item
If $i_l= k-1$, then $\partial_l^2 \wC_{\bi}(E) =
\wC_{\partial_l(\bi)}(\partial_l^2E)$ and $\partial_l^1
\wC_{\bi}(E)= \wC_{\partial_l(\bi)}(\partial_l^1E)$. \item Lemma
\ref{tens} remains valid for the cubes $\wC_{\bi}(E)$.
\end{enumerate}
\end{prop}
\begin{flushright}
 $\square$
 \end{flushright}

\begin{obs} Let $(E,f)$ be a split $n$-cube.
Observe that by the choice of isomorphisms, for every $\bj$ with $j_l=1$ and
$\bi$ with $i_l=k-1$, the arrows
$$\wC_{\bi}(E)^{\sigma_l^0(\bj)} \rightarrow \wC_{\bi}(E)^{\bj}
\quad\textrm{and}\quad\wC_{\bi}(E)^{\bj} \rightarrow
\wC_{\bi}(E)^{\sigma_l^2(\bj)}$$ are induced by the arrows
$$\partial_l^0E \rightarrow \partial_l^1E,\quad \partial_l^1E \rightarrow \partial_l^2E
 \quad\textrm{and}\quad \partial_l^2E\rightarrow \partial_l^0E\oplus \partial_l^2E \xrightarrow{f} \partial_l^1E.$$
 \end{obs}
 
 \subsection{Definition of $\Psi^k$}
In this section, we define a morphism
$$\Z \Sp\nolimits_n(X) \xrightarrow{\Psi^k}  \Z SG^k(X)_n $$
using the cubes $\wC_{\bi}(E)$ constructed in the previous
section.

Recall that when $i_l\neq k-1$, the exact sequence
$\partial_l^2\wC_{\bi}(E)$ is canonically isomorphic to the simple associated
to a $2$-iterated cochain complex  of $(n-1)$-cubes, of lengths
$(k-i_l-1,i_l+1)$. We define
\begin{eqnarray*}
\Z \Sp\nolimits_n(X) &\xrightarrow{\Psi^k} & \Z SG^k(X)_n =
\bigoplus_{m=1}^{k-1}\Z SG^{m,k}_2(X)_{n-1}\oplus
\Z SG^k_1(X)_{n} \\
(E,f) &\mapsto &
(\Psi^{1,k}_2(E),\dots,\Psi^{k-1,k}_2(E),\Psi^k_1(E))
\end{eqnarray*}
with
\begin{eqnarray*}
\Psi^k_1(E) &=&\sum_{\bi\in [0,k-1]^n}\wC_{\bi}(E), \\
\Psi^{m,k}_2(E) &=& \sum_{l=1}^n(-1)^{m+l+1}\sum_{\substack{\bi\in [0,k-1]^n,\\
i_l=m-1}} \partial_{l}^2\wC_{\bi}(E),\quad \textrm{for
}m=1,\dots,k-1,
\end{eqnarray*}
where in the last equality we consider, by proposition
\ref{partials} $(iv)$, $\partial_{l}^2\wC_{\bi}(E)$ as a
$2$-iterated complex.

\begin{obs}\label{iso2}
Observe that in the last definition, considering
$\partial_{l}^2\wC_{\bi}(E)$ as a tensor product of complexes
involves changing the order in the summations. To be precise,
recall that the terms corresponding to the indices $\bn$ with
$n_l=0$ form the first complex in the tensor product, and the ones
with $n_l=1$ form the second complex. The tensor product of the
Koszul complexes in each summand of $\wC_{\bi}(E)$ is ordered by
the lexicographic order. Hence, the two orders would only agree
for $l=1$. For instance,
\begin{enumerate*}[$\rhd$]
\item the face $\partial_2^2$ of the cube $C_{00}(E)$ (see example
$k=2,n=2$), is $$[\Psi^1(E^{00})\otimes
\Psi^1(E^{02}),\Psi^1(E^{00})\otimes \Psi^1(E^{22})\oplus
\Psi^1(E^{02})\otimes \Psi^1(E^{20})],$$ \item this complex,
viewed as a tensor product of complexes, is {\small $$ \setlength\arraycolsep{-1pt}
\left[\begin{array}{c} \xymatrix@C=0.4cm{ E^{00}\otimes E^{02}
\ar[r]\ar[d] & E^{00}\otimes E^{02} \ar[d] \\ E^{00}\otimes E^{02}
\ar[r] & E^{00}\otimes E^{02}
}\end{array},\begin{array}{c}\xymatrix@C=0.5cm{E^{00}\otimes
E^{22} \oplus E^{20}\otimes E^{02}  \ar[r]\ar[d] & E^{00}\otimes
E^{22} \oplus E^{20}\otimes E^{02}  \ar[d] \\ E^{00}\otimes E^{22}
\oplus E^{20}\otimes E^{02} \ar[r] & E^{00}\otimes E^{22} \oplus
E^{20}\otimes E^{02} } \end{array}\right] $$}
\end{enumerate*}
Notice the difference in the order of $\Psi^1(E^{20})\otimes
\Psi^1(E^{02})$.

Hence, strictly speaking, $\Psi^k$ cannot be a chain morphism.
However, for every split $n$-cube, the composition of $\Psi^k$
with $\varphi$ leads to a collection of $n$-cubes (see the
definition of $\varphi$ in \eqref{morphismvarphi}). The locally
free sheaves of these cubes are direct sums, tensor products,
exterior products and symmetric products of the locally free
sheaves $E^{\bj}$.

We can map, with the corresponding canonical isomorphism, every
$n$-cube of $\varphi\circ\Psi^k(E)$ to the $n$-cube whose summands
are all ordered by the lexicographic order. Then,
$\varphi\circ\Psi^k$ is a chain morphism.

This trick can only be performed after the composition with
$\varphi$, and cannot be corrected in the definition of $\Psi^k$.
\end{obs}

\begin{prop} Let $E$ be a split $n$-cube. Then, there is a canonical isomorphism
$$d_s \Psi^k(E) \cong_{can} \Psi^k(dE). $$
\end{prop}
\begin{proof}
We have to see that
\begin{eqnarray}
\Psi^{m,k}_2(dE) &=& -d\Psi^{m,k}_2(E),\quad \textrm{for }m=1,\dots,k-1, \label{eq3} \\
\Psi^k_1(dE)&=& \sum_{m=1}^{k-1}(-1)^m\Phi^m
(\Psi^{m,k}_2(E))+d\Psi^k_1(E).\label{eq4}
\end{eqnarray}

We start by proving \eqref{eq4}. By definition,
$$
d\Psi^k_1(E) =  \sum_{\bi\in [0,k-1]^n}d\wC_{\bi}(E) =\sum_{\bi\in
[0,k-1]^n}\sum_{l=1}^n\sum_{s=0}^2 (-1)^{l+s}
\partial_l^s\wC_{\bi}(E). $$

Then, by proposition \ref{partials}, {\small \begin{eqnarray*}
d\Psi^k_1(E)& \cong_{can} & \sum_{l=1}^n (-1)^l \Bigg[\sum_{\bi\in
[0,k-1]^n, i_l=0}
\wC_{\partial_l(\bi)}(\partial_l^0E)-\sum_{\substack{\bi\in
[0,k-1]^n,\\
i_l=k-1}} \wC_{\partial_l(\bi)}(\partial_l^1E) \\
&& \hspace{1.3cm} +\sum_{\substack{\bi\in [0,k-1]^n \\ i_l=k-1}}
\wC_{\partial_l(\bi)}(\partial_l^2E)+
\sum_{m=0}^{k-2}\sum_{\substack{\bi\in [0,k-1]^n\\
i_l=m}}\Phi^{i_l+1}(\partial_l^2\wC_{\bi}(E)) \Bigg].
\end{eqnarray*}}
 Therefore,
{\small \begin{eqnarray*}
 d\Psi^k_1(E) & \cong_{can}& \sum_{l=1}^n \sum_{s=0}^2
(-1)^{l+s} \sum_{\bi\in [0,k-1]^{n-1}} \wC_{\bi}(\partial_l^sE) \\ && +
\sum_{l=1}^n (-1)^l \sum_{m=0}^{k-2}\sum_{\substack{\bi\in
[0,k-1]^n,\\i_l=m}}\Phi^{m+1}(\partial_l^2\wC_{\bi}(E))
\\
&=& \Psi^k_1(dE) -\sum_{m=1}^{k-1}(-1)^m\Phi^m (\Psi^{m,k}_2(E)),
\end{eqnarray*}}
and equality \eqref{eq4} is proven. Let us prove now \eqref{eq3}.
First of all observe that due to the alternating sum of $\partial_l^2$ in the definition of 
$\Psi^{m,k}_2$, we have
$$\sum_{r=1}^{n-1}(-1)^r \partial_r^2\Psi^{m,k}_2(E)=0. $$
By the same argument, considering the first equality of $(iii)$ in proposition \ref{partials}, we have 
$$\sum_{r=1}^{n}(-1)^r\Psi^{m,k}_2( \partial_r^2 E)=0. $$
Therefore, we are left to see that
$$\sum_{r=1}^{n}(-1)^r(\Psi^{m,k}_2( \partial_r^0 E)-\Phi^{m,k}_2( \partial_r^1 E))=-\sum_{r=1}^{n-1}(-1)^r( \partial_r^0- \partial_r^1)\Psi^{m,k}_2( E).$$

Recall that for all $j$, $\quad
\partial_r^j\partial_{l}^1= \left\{\begin{array}{ll}
\partial_l^1\partial_{r+1}^j, & \textrm{if }r\geq l, \\
\partial_{l-1}^1\partial_{r}^j, & \textrm{if }r < l. \end{array}
 \right.$

Hence, we split the following expression accordingly:
$$\sum_{r=1}^{n-1}(-1)^r  \partial_r^0 \Psi^{m,k}_2( E) =(A)+(B),$$
with:
\begin{eqnarray*}
  (A) &=& \sum_{l=2}^n\sum_{r=1}^{l-1} (-1)^{m+l+r+1} \sum_{\substack{\bi\in [0,k-1]^n\\
i_l=m-1}}  \partial_r^0 \partial_{l}^2\wC_{\bi}(E). \\
   (B) &=& \sum_{l=1}^{n-1}\sum_{r=l}^{n-1} (-1)^{m+l+r+1} \sum_{\substack{\bi\in [0,k-1]^n\\
i_l=m-1}}  \partial_r^0 \partial_{l}^2\wC_{\bi}(E).
  \end{eqnarray*}
  
  Using the mentioned equalities for the faces and the identities $(i)-(iii)$ of proposition \ref{partials}, we have:
  \begin{eqnarray*}
  (A) &=& \sum_{l=2}^n\sum_{r=1}^{l-1} (-1)^{m+l+r+1} \sum_{\substack{\bi\in [0,k-1]^n\\
i_l=m-1}}  \partial_{l-1}^2\partial_r^0\wC_{\bi}(E) \\
&=& \sum_{l=2}^n\sum_{r=1}^{l-1} (-1)^{m+l+r+1} \sum_{\substack{\bi\in [0,k-1]^n\\
i_l=m-1 \\ i_r\neq 0}}  \partial_{l-1}^2\partial_r^1\wC_{\bi-1_{l}}(E) +
\sum_{\substack{\bi\in [0,k-1]^n\\
i_l=m-1 \\ i_r=0}} \partial_{l-1}^2\wC_{\partial_l(\bi)}(\partial_r^0 E) \\
&=& \sum_{l=2}^n\sum_{r=1}^{l-1} (-1)^{m+l+r+1} \sum_{\substack{\bi\in [0,k-1]^n\\
i_l=m-1 }}  \partial_r^1\partial_{l}^2\wC_{\bi}(E) - \sum_{\substack{\bi\in [0,k-1]^n\\
i_l=m-1 \\ i_r= k-1}}  \partial_{l-1}^2\partial_{r}^1\wC_{\bi}( E) + \\ && \sum_{l=1}^{n-1}\sum_{r=1}^{l} (-1)^{m+l+r} 
\sum_{\substack{\bi\in [0,k-1]^{n-1}\\
i_{l}=m-1 }}  \partial_{l}^2\wC_{\bi}(\partial_r^0 E) \\
&=& \sum_{l=2}^n\sum_{r=1}^{l-1} (-1)^{m+l+r+1} \sum_{\substack{\bi\in [0,k-1]^n\\
i_l=m-1 }}  \partial_r^1\partial_{l}^2\wC_{\bi}(E) + \\
&&\sum_{l=1}^{n-1}\sum_{r=1}^{l} (-1)^{m+l+r}
\sum_{\substack{\bi\in [0,k-1]^{n-1}\\ i_{l}=m-1 }}  \partial_{l}^2\wC_{\bi}(\partial_r^0 E) -\partial_{l}^2\wC_{\bi}(\partial_r^1 E) 
\end{eqnarray*}
Reasoning analogously, we obtain
 \begin{eqnarray*}
  (B) &=& \sum_{l=1}^{n-1}\sum_{r=l}^{n-1} (-1)^{m+l+r+1} \sum_{\substack{\bi\in [0,k-1]^n\\
i_l=m-1 }}  \partial_r^1\partial_{l}^2\wC_{\bi}(E) + \\
&&\sum_{l=1}^{n-1}\sum_{r=l+1}^{n} (-1)^{m+l+r}
\sum_{\substack{\bi\in [0,k-1]^{n-1}\\ i_{l}=m-1 }}  \partial_{l}^2\wC_{\bi}(\partial_r^0 E) -\partial_{l}^2\wC_{\bi}(\partial_r^1 E)
\end{eqnarray*}

It follows that
\begin{eqnarray*}
  \sum_{r=1}^{n-1}(-1)^r  \partial_r^0 \Psi^{m,k}_2( E)  &=&  \sum_{r=1}^{n-1}(-1)^r  \partial_r^1 \Psi^{m,k}_2( E)  - \sum_{r=1}^{n}(-1)^r \big(\partial_{l}^2\wC_{\bi}(\partial_r^0 E) -\partial_{l}^2\wC_{\bi}(\partial_r^1 E)  \big)\\ &=&
  \sum_{r=1}^{n-1}(-1)^r  \partial_r^1 \Psi^{m,k}_2( E)  - \sum_{r=1}^{n}(-1)^r  \big( \Psi_2^{k,m}(\partial_r^0E)-\Psi_2^{k,m}(\partial_r^1E)\big)
\end{eqnarray*}
as desired.
\end{proof}

For every $n$, let
$$\Psi^k:\Z \Spn_n(X) \rightarrow \Z C_n(X)$$ be the composition
$$ \Psi^k: \Z\Spn_n(X) \xrightarrow{\Psi^k} \Z SG^k(X)_n \xrightarrow{\mu\circ \varphi} \Z C_*(X),$$
modified by the canonical isomorphisms, so that every direct sum
of tensor, exterior and symmetric products is ordered by the
lexicographic order of the corresponding multi-indices.

\begin{cor}\label{adamssplitcor} For every scheme $X$, there is a well-defined chain morphism
  $$\Psi^k : \Z \Sp\nolimits _*(X) \rightarrow \Z C_n(X).$$
\end{cor}
\begin{flushright}
  $\square$
\end{flushright}

\begin{ex}[$\mathbf{k=2,n=1}$] Let us have a look at the situation with $k=2$ and $n=1$, as considered in the example cases (see section \ref{adams2}). Let $(E,f): E^0\rightarrow E^1 \rightarrow E^2$ be a split cube (with splitting $f$). In the complex $\Z SG^k(X)_*$, the image of $(E,f)$ by $\Psi^2$ is
  $$(\Psi^1(E^0)\otimes
\Psi^1(E^2), \widetilde{C}_0(E,f)+\widetilde{C}_1(E,f))\in \Z G_2^{1,2}(X)_0 \oplus \Z G_1^2(X)_1.$$
  Applying $\varphi$ to this element, we get the following linear combination of short exact sequences
  \begin{align*}
-3(E^0\otimes E^2 \rightarrow E^0\otimes E^2) + (E^0\otimes E^2 \rightarrow E^0\otimes E^2\oplus E^0\otimes E^2 \rightarrow E^0\otimes E^2)
 -\\ -2(\wedge^2E^0  \rightarrow \wedge^2E^0 \oplus E^0\otimes E^2\rightarrow E^0\otimes E^2)-2(\wedge^2E^0 \oplus E^0\otimes E^2 \rightarrow\wedge^2E^1  \rightarrow \wedge^2E^2) + \\ +(E^0\otimes E^0  \rightarrow E^0\otimes E^0 \oplus E^0\otimes E^2\oplus E^0\otimes E^2\rightarrow E^0\otimes E^2\oplus E^0\otimes E^2) + \\ +(E^0\otimes E^0 \oplus E^0\otimes E^2\oplus E^0\otimes E^2 \rightarrow E^1\otimes E^1 \rightarrow E^2\otimes E^2).
 \end{align*} 
By the action of $\mu$, a short exact sequence $A\rightarrow B \xrightarrow{g} C$ transforms into 
$$-(0\rightarrow A \rightarrow \ker g)+(\ker g \rightarrow B \rightarrow C).$$
 Therefore, the explicit expression of $\Psi^2(E,f)$ is obtained by transforming by $\mu$ each of the short exact sequences of the linear combination detailed above.
 
 One easily checks that the differential of $\Psi^2(E,f)$ is $-2(-\wedge^2E^0 +\wedge^2E^1 -\wedge^2E^2) + (-E^0\otimes E^0 + E^1\otimes E^1 -E^2\otimes E^2)$ which is exactly $-\Psi^2(E^0) + \Psi^2(E^1)-\Psi^2(E^2)$ as desired.

\end{ex}

\section{The transgression morphism}\label{transgression}
Fix $\mathcal{C}_{B}$ to be a category of  schemes over a  base
scheme $B$. In this section, we introduce all the ingredients for
the definition of Adams operations on the rational algebraic
$K$-theory of a regular noetherian scheme.

Let $X$ be a scheme. We
first define a chain complex $\wZ C_*^{\widetilde{\square}}(X)$
that is the target for the Adams operations. Then, we prove that
it is quasi-isomorphic   to the chain complex of cubes with
rational coefficients. Hence, its rational homology groups are
isomorphic to the rational $K$-groups. Finally, we define a
morphism, \emph{the transgression morphism}, from $\Z C_*(X)$ to a
new chain complex $\Z \Spn_*^{\square}(X)$, whose image consists
only of split cubes. Then, for each $k$, the morphism $\Psi^k$
defined in the previous section induces a morphism
$$\Psi^k:\Z \Spn_*^{\square}(X)\rightarrow \wZ C_*^{\widetilde{\square}}(X).$$
Composing with the transgression morphism we obtain a chain
complex (denoted, by abuse of notation, by $\Psi^k$):
$$\Psi^k:NC_*(X)\rightarrow \Z C_*(X) \xrightarrow{T} \Z\Spn_*^{\square}(X)  \rightarrow \wZ C_*^{\widetilde{\square}}(X).$$

\subsection{The transgression chain complex}\label{projectiucocubical}
 Let $\P^1=\P^1_B$ be the
projective line over the base scheme $B$ and let
$$ \square = \P^1\setminus \{1\}\cong \A^1.$$
The cartesian product $(\P^1)^{\cdot}$ has a cocubical scheme structure.
Specifically, the face and degeneracy maps
\begin{eqnarray*}
\delta_j^i : (\P^1)^n &\rightarrow & (\P^1)^{n+1},\quad i=1,\dots,n,\ j=0,1, \\
\sigma^i : (\P^1)^n &\rightarrow & (\P^1)^{n-1} ,\quad
i=1,\dots,n,
\end{eqnarray*}
are defined as
\begin{eqnarray*}
  \delta_0^i(x_1,\dots,x_n) &=& (x_1,\dots,x_{i-1},(0:1),x_{i},\dots,x_n), \\
\delta_1^i(x_1,\dots,x_n) &=& (x_1,\dots,x_{i-1},(1:0),x_{i},\dots,x_n), \\
\sigma^i(x_1,\dots,x_n) &=& (x_1,\dots,x_{i-1},x_{i+1},\dots,x_n).
\end{eqnarray*}
These maps satisfy the usual
identities for a cocubical object in a category, and leave
invariant $\square$. Hence, both $(\P^1)^{\cdot}$ and $
\square^{\cdot}$ are cocubical schemes.

Let $X\times (\P^1)^n$ and $X\times \square^n$ denote $X\times _{B}(\P^1)^n$
and $X\times _{B}\square^n$ respectively. Since most of the constructions will
be analogous for $\P^1$ and for $\square$, we write
$$\B = \P^1 \textrm{ or }\ \square. $$

For $i=1,\dots,n$ and $j=0,1$, consider the chain morphisms induced on the
complex of cubes
\begin{eqnarray*}
\delta^j_i=(Id\times \delta_j^i)^* & : & \Z C_*(X\times \B^n)  \rightarrow  \Z
C_*(X\times \B^{n-1}),
\\ \sigma_i=(Id\times
\sigma^i)^* & : &  \Z C_*(X\times \B^{n-1})  \rightarrow  \Z C_*(X\times
\B^{n}).
\end{eqnarray*}
These maps endow $\Z C_*(X\times \B^{\cdot})$ with a cubical chain
complex structure. Observe that if $(x_1:y_1),\dots,(x_n:y_n)$ are
homogeneous coordinates of $(\P^1)^n$, then $\delta_i^0$
corresponds to the restriction map  to the hyperplane $x_i=0$ and
$\delta_i^1$ corresponds to the restriction map to the hyperplane
$y_i=0$. On the affine lines, with coordinates $(t_1,\dots,t_n)$
where $t_i=\frac{x_i}{y_i}$, the map $\delta_i^0$ corresponds
to the restriction map to the hyperplane $t_i=0$ and the map
$\delta_i^1$ to the restriction map to the hyperplane $t_i=\infty$.

Let  $\Z C^{\B}_{*,*}(X)$ be the 2-iterated chain complex given by
$$\Z C^{\B}_{r,n}(X): = \Z C_{r}(X\times\B^{n}),$$ and differentials
$$d = d_{C_*(X\times\B^{n})},\qquad \delta = \sum (-1)^{i+j}\delta^j_i.$$
Denote by $(\Z C^{\B}_{*}(X),d_s)$ the associated simple complex.

Observe that, by functoriality, the face and degeneracy maps
$\partial_i^j$ and $s_i^j$, as defined in sections \ref{iterated} and \ref{cubessection}, commute with $\delta_i^j$ and $\sigma_i$.
Therefore, there are analogous $2$-iterated chain complexes
\begin{eqnarray*}
\wZ C^{\B}_{r,n}(X)&:=& \Z C_{r}(X\times \B^n) / \Z D_{r}(X\times \B^n), \\
NC^{\B}_{r,n}(X)&:=& N C_{r}(X\times \B^n).
\end{eqnarray*}
Recall from section \ref{sectionnormalized} that $N C_{*}(X\times \B^n)$ is the normalized complex of
cubes in $X\times \B^n$ and from section \ref{cubessection} that $\Z D_{*}(X\times \B^n)$ is the
complex of degenerate cubes in $X\times \B^n$. That is, we consider the normalized complex of cubes and the
quotient by degenerate cubes to the first direction of the $2$-iterated complex
$\Z C^{\B}_{r,n}(X)$.

In addition, the second direction of these $2$-iterated complexes
corresponds to the chain complex associated to a cubical abelian
group. Therefore, one has to factor out by the degenerate
elements.

Let $\Z C_*(X\times \square^n)_{deg}\subset \Z C_*(X\times
\square^n) $  be the subcomplex consisting of the degenerate
elements, i.e. that lie in the image of $\sigma_i$ for some $i=1,\dots,n$.
Analogously, we define the complexes
$$NC^{\square}_{r,n}(X)_{\deg}= N C^{\square}_{r,n}(X)\cap \Z
C_r(X\times \square^n)_{deg},$$ and
$$\wZ C^{\square}_{r,n}(X)_{\deg}= \Z
C_r(X\times \square^n)_{deg}/ \Z D_{r}(X\times \square^n)_{deg}$$
of degenerate elements in $NC_{r,n}^{\square}(X)$ and $\wZ C_{r,n}^{\square}(X)$ respectively.

We define
the $2$-iterated chain complexes,
\begin{eqnarray*}
\wZ C^{\widetilde{\square}}_{r,n}(X) &:=&   \wZ C^{\square}_{r,n}(X)/\wZ C^{\square}_{r,n}(X)_{\deg}, \\
N C^{\widetilde{\square}}_{r,n}(X) &:=&   N C^{\square}_{r,n}(X)/ N
C^{\square}_{r,n}(X)_{deg}.
\end{eqnarray*}
Denote by $(\wZ
C^{\widetilde{\square}}_*(X),d_s)$ and $(N
C^{\widetilde{\square}}_*(X),d_s)$ the simple complexes associated
to these 2-iterated chain complexes.

\begin{prop}\label{affine}
If $X$ is a regular noetherian scheme, the natural morphism of complexes
$$N C_*(X)= N C^{\widetilde{\square}}_{*,0}(X) \rightarrow NC^{\widetilde{\square}}_*(X)$$
induces an isomorphism in homology groups with coefficients in
$\Q$.
\end{prop}
\begin{proof} Consider the first quadrant spectral sequence with $E^0$ term given by
$$E^0_{r,n}= N C^{\widetilde{\square}}_{r,n}(X)\otimes \Q.$$
When it converges, it converges to the homology groups $H_*(N
C^{\widetilde{\square}}_*(X),\Q)$. If we see that for all $n>0$
the rational homology of the complex $N
C^{\widetilde{\square}}_{*,n}(X)$ is zero, the spectral sequence
converges and the proposition is proven.

This is proved by an induction argument. For every $j=1,\dots,n$,
let
$$NC^{\square,j}_{r,n}(X)_{deg}=  \sum_{i=1}^j \sigma_i(N C^{\square}_{*,n-1}(X)) \subseteq
N C^{\square}_{r,n}(X)_{deg}$$  and let $N
C^{\widetilde{\square},j}_{*,n}(X)$ be the respective quotient. We
will show that, for all $n>0$ and $j=1,\dots,n$,
$$H_*(N C^{\widetilde{\square},j}_{*,n}(X),\Q)=0.$$
For $j=1$ and $n>0$,
$$N C^{\widetilde{\square},1}_{*,n}(X)=N C^{\square}_{*,n}(X)/ \sigma_1(N C^{\square}_{*,n-1}(X)).$$
By the homotopy invariance of algebraic $K$-theory of regular
noetherian schemes, the rational homology of this complex is zero.
Then, if $j>1$ and $n>1$, $N C^{\widetilde{\square},j}_{*,n}(X)$
is the cokernel of the monomorphism
 \begin{eqnarray*}
 N C^{\widetilde{\square},j-1}_{*,n-1}(X)  &\xrightarrow{\sigma_j}& N C^{\widetilde{\square},j-1}_{*,n}(X) \\
 E &\mapsto & \sigma_j(E).
\end{eqnarray*}
Since by the induction hypothesis both sides have zero rational homology, so
does the cokernel.
\end{proof}

Observe that in the proof of last proposition, the key point was that for
regular noetherian schemes, the  $K$-groups of $X\times \square^n$ are
isomorphic to the $K$-groups of $X$.  In the case of projective lines, the
situation is slightly trickier, because the $K$-groups of $X\times \P^1$
are not isomorphic to the $K$-groups of $X$. We have to use  the Dold-Thom
isomorphism relating both groups. This implies that we shall also factor out by
the class of the canonical bundle on $(\P^1)^n$.

Let $p_1,\dots,p_n$ be the projections onto the $i$-th coordinate of
$(\P^1)^n$. Consider the invertible sheaf
$\mathcal{O}(1):=\mathcal{O}_{\P^1}(1)$, the dual of the tautological sheaf of
$\P^1$, $\mathcal{O}_{\P^1}(-1)$. We define then the 2-iterated chain complexes
\begin{eqnarray*}
N C^{\P}_{r,n}(X)_{deg} &:=& \sum_{i=1}^n  \sigma_i(N
C^{\P}_{r,n-1}(X)) + p_i^*\mathcal{O}(1) \otimes \sigma_i(N
C^{\P}_{r,n-1}(X)),
 \\
N C^{\widetilde{\P}}_{r,n}(X) &:=&   N C^{\P}_{r,n}(X)/N C^{\P}_{r,n}(X)_{deg}.
\end{eqnarray*}
Denote by $(N C^{\widetilde{\P}}_*(X),d_s)$  the simple complex associated
to this 2-iterated chain complex.

\begin{prop}\label{dold} If $X$ is a regular noetherian scheme,
the natural morphism of complexes
$$N C_*(X)= N C^{\widetilde{\P}}_{*,0}(X) \rightarrow N C^{\widetilde{\P}}_*(X)$$
induces an isomorphism on homology with coefficients in $\Q$.
\end{prop}
\begin{proof} The proof is analogous to the proof of the last proposition. By
considering the spectral sequence associated with the homology of a
$2$-iterated complex, we just have to see that for all $j$,
$$H_*(N C^{\widetilde{\P},j}_{*,n}(X),\Q)=0.$$
For $j=1$ and $n>0$, it follows from the Dold-Thom isomorphism on algebraic $K$-theory of
regular noetherian schemes. For $j>1$ and $n>1$, $N
C^{\widetilde{\P},j}_{*,n}(X)$ is the cokernel of the monomorphism
 \begin{eqnarray*}
 N C^{\widetilde{\P},j-1}_{*,n-1}(X) \oplus N \C^{\widetilde{\P},j-1}_{*,n-1}(X)
 &\rightarrow & N C^{\widetilde{\P},j-1}_{*,n}(X) \\
 (E_0,E_1) &\mapsto & \sigma_j(E_0) + p_j^*\mathcal{O}(1)\otimes \sigma_j(E_1).
\end{eqnarray*}
Since by the induction hypothesis both sides have zero rational homology, so
does the cokernel.
\end{proof}

\begin{obs}\label{projectiuquocient}
Let
\begin{eqnarray*}
\Z C^{\P}_{r,n}(X)_{deg} & = & \sum_{i=1}^n  \sigma_i(\Z
C^{\P}_{r,n-1}(X)) + p_i^*\mathcal{O}(1) \otimes \sigma_i(\Z
C^{\P}_{r,n-1}(X)), \\
\wZ C^{\P}_{r,n}(X)_{\deg} & = & \Z C^{\P}_{r,n}(X)_{deg}/ \Z
D_{r}(X\times (\P^1)^n)_{deg},
\end{eqnarray*} and
let
$$\wZ C^{\widetilde{\P}}_{r,n}(X) :=   \wZ C^{\P}_{r,n}(X)/\wZ C^{\P}_{r,n}(X)_{\deg}.$$
Denote by $(\wZ C^{\widetilde{\P}}_*(X),d_s)$ the simple complex
associated to this 2-iterated chain complex.

It follows from the definitions that there is an isomorphism
$$\wZ C^{\widetilde{\P}}_*(X) \cong N C^{\widetilde{\P}}_*(X). $$
\end{obs}

\subsection{The transgression of cubes by affine and projective lines}\label{transgressionsection}
Let $x$ and $y$ be the global sections of $\mathcal{O}(1)$ given
by the projective coordinates $(x:y)$ on $\P^1$. Let $X$ be a
scheme and let  $p_0$ and $p_1$ be the projections from $X\times
\P^1$ to $X$ and $\P^1$ respectively. Then, for every locally free
sheaf $E$ on $X$, we denote
$$E(k):=p_0^* E\otimes p_1^*\mathcal{O}(k).$$

The following definition is taken from \cite{Burgos1}.
\begin{df}\label{definiciotransgressio} Let
$$E: 0\rightarrow E^0 \xrightarrow{f^0} E^1 \xrightarrow{f^1} E^2 \rightarrow 0$$
be a short exact sequence. The \emph{first transgression by
projective lines} of $E$, $\tr_1(E)$, is the kernel of the
morphism
\begin{eqnarray*}
E^1(1) \oplus E^2(1) &\rightarrow & E^2(2) \\ (a,b) &\mapsto &
f^1(a)\otimes x-b\otimes y.
\end{eqnarray*}
\end{df}
Observe that this locally free sheaf on $X\times \P^1$ satisfies that
\begin{eqnarray*}
\delta_1^0\tr\nolimits_1(E)& = & \tr\nolimits_1(E)|_{x=0} =  E^1, \\
\delta_1^1\tr\nolimits_1(E)&=& \tr\nolimits_1(E)|_{y=0} = \im f^0
\oplus E^2.
\end{eqnarray*}

By restriction to $\square$, we obtain the \emph{transgression by affine
lines}.

From now on, we restrict ourselves to the affine case. However,
all the results can be written in terms of the complexes with
projective lines.

Let $E$ be an $n$-cube. We define  the \emph{first transgression}
of $E$ as the $(n-1)$-cube on $X\times \square^1$ given by
$$ \tr\nolimits_1(E)^{\bj}:=
\tr\nolimits_1(\partial_{2,\dots,n}^{\bj}E),\qquad \textrm{for all } \bj\in
 \{0,1,2\}^{n-1},$$
i.e. we take the transgression of the exact sequences in the first
direction. Since $\tr\nolimits_1$ is a functorial exact
construction, the \emph{$m$-th transgression sheaf}
can be defined recursively as
$$\tr\nolimits_m(E)=\tr\nolimits_1\tr\nolimits_{m-1}(E)=
\tr\nolimits_1\stackrel{m}{\dots}\tr\nolimits_1(E).$$ It is an
$(n-m)$-cube on $X\times \square^{m}$. In particular, $\tr_n(E)$
is a locally free sheaf on $X\times \square^n$.

Observe that the transgression is functorial, i.e., if
$E\xrightarrow{\psi} F$ is a morphism of $n$-cubes, then there is
an induced morphism
$$ \tr\nolimits_m(E) \xrightarrow{\tr\nolimits_{m}(\psi)}
\tr\nolimits_m(F),$$ for every $m=1,\dots,n$. In particular, for
every $n$-cube $E$ and $i=1,\dots,n$, the morphism of
$(n-1)$-cubes
$$\partial_i^0 E\xrightarrow{f_i^0} \partial_i^1 E,$$
induces a morphism
$$ \tr\nolimits_m(\partial_i^0 E)\xrightarrow{\tr\nolimits_m(f_i^0)} \tr\nolimits_m(\partial_i^1 E),$$
for $m=1,\dots,n-1$.

\begin{lema}\label{trans}
For every $n$-cube $E$ and $i=1,\dots,n$, the following identities
hold:
\begin{eqnarray}
\delta_i^0\tr\nolimits_n(E)&=&\tr\nolimits_{n-1}(\partial_i^1E), \\
\delta_i^1\tr\nolimits_n(E)&\cong& \im \tr\nolimits_{n-1}(f_i^0)
\oplus\tr\nolimits_{n-1}(\partial_i^2E),\label{iso}
\end{eqnarray}
where the isomorphism is canonical, i.e. a combination of
commutativity and associativity isomorphisms for direct sums,
distributivity isomorphism for the tensor product of a direct sum and commutativity isomorphisms of the pull-back of
a direct sum with the direct sum of the pull-back.
\end{lema}
\begin{proof}
The proof is straightforward.  For $n=1$ it follows from the definition.
Therefore,
\begin{eqnarray*}
\delta_i^0\tr\nolimits_n(E) &=&
\delta_i^0\tr\nolimits_1\stackrel{n}{\dots}\tr\nolimits_1(E)  \\ &=&
\tr\nolimits_{n-i}\partial_i^1\tr\nolimits_{i-1}(E)
=\tr\nolimits_{n-1}(\partial_i^1E).
\end{eqnarray*}
For the second statement, observe first of all that there is a canonical
isomorphism $\tr\nolimits_n(A\oplus B)\cong \tr\nolimits_n(A)\oplus
\tr\nolimits_n(B)$. It follows by recurrence  from the case $n=1$. So, let
$E,F$ be two short exact sequences. Then, $\tr\nolimits_1(E\oplus F)$ is the
kernel of the map
$$(E_1 \oplus F_1)(1) \oplus (E_2\oplus F_2)(1) \rightarrow (E_2\oplus F_2)(2), $$
while $\tr\nolimits_1(E)\oplus \tr\nolimits_1(F)$ is the direct sum of the
kernels of the maps
$$E_1(1) \oplus E_2(1) \rightarrow E_2(2),\quad  F_1(1) \oplus F_2(1) \rightarrow F_2(2).  $$
Hence, there is clearly a canonical isomorphism.  Therefore
\begin{eqnarray*}
\delta_i^1\tr\nolimits_n(E)&=&
\delta_i^1\tr\nolimits_1\stackrel{n}{\dots}\tr\nolimits_1(E)  \\ &
= & \tr\nolimits_{n-i}(\im\tr\nolimits_{i-1}(f_i^0)\oplus
\tr\nolimits_{i-1}(\partial_i^2E)) \\ &\cong &\im
\tr\nolimits_{n-1}(f_i^0)\oplus \tr\nolimits_{n-1}(\partial_i^2E).
\end{eqnarray*}
\end{proof}
\begin{obs}
Since the isomorphisms in \eqref{iso} are  canonical,  if
$E\xrightarrow{\psi} F$ is a morphism of $n$-cubes,  for every $i$  we obtain a
commutative diagram
$$\xymatrix{
\delta_i^1\tr\nolimits_{n}(E)
\ar[r]^{\delta_i^1\tr\nolimits_n(\psi)} \ar[d]_{\cong} &
\delta_i^1\tr\nolimits_{n}(F) \ar[d]^{\cong}
\\ \im \tr\nolimits_{n-1}(f_i^0)
\oplus\tr\nolimits_{n-1}(\partial_i^2E) \ar[r]  & \im
\tr\nolimits_{n-1}(f_i^0) \oplus\tr\nolimits_{n-1}(\partial_i^2F).
 }$$
\end{obs}

Observe that for every $n$-cube $E$, the $(n-m)$-cube
$\tr\nolimits_m(E)$ is obtained applying $\tr_1$ on the directions
$1,\dots,m$. In the next definition, we generalize this
construction by specifying the directions in which we apply the first
transgression.

\index{transgression!sheaf}
\begin{df}
Let $\bi=(i_1,\dots,i_{n-m})$, with $1\leq i_1<\dots < i_{n-m}\leq n$. We
define, $\tr\nolimits_{m}^{\bi}(E)\in C^{\square}_{n-m,m}(X)$, by
$$\tr\nolimits_{m}^{\bi}(E)^{\bj} =
 \tr\nolimits_{m}(\partial_{\bi}^{\bj}E),\qquad \textrm{for all }
\bj\in \{0,1,2\}^{n-m}.$$
\end{df}
In other words, we consider the first transgression
iteratively in the directions not in the multi-index $\bi$, from
the highest  to the lowest index.

By convention, the transgressions without super-index correspond
to the multi-index $\bi=(m+1,\dots,n)$. The following lemma is a
direct
consequence of  lemma \ref{trans}.

\begin{lema}\label{trans3}
Let $\bi=(i_1,\dots,i_{n-m})$ with $1\leq i_1<\dots < i_{n-m}\leq n$. Then, for
every fixed $r=1,\dots,m$, let $\bi'=\bi-\bun_{r+1}^{n-m}$ and let
$(v_1,\dots,v_m)$ be the ordered multi-index with the entries in
$\{1,\dots,n\}\setminus \{i_1,\dots,i_{n-m}\}$. Then,
\begin{eqnarray*}
\delta_r^0\tr\nolimits^{\bi}_m(E) &=&\tr\nolimits_{m-1}^{\bi'}
(\partial_{v_r}^1E), \\
\delta_r^1\tr\nolimits_m^{\bi}(E) &\cong& \im
\tr\nolimits_{m-1}^{\bi'}(f_{v_r}^0)
\oplus\tr\nolimits_{m-1}^{\bi'}(\partial_{v_r}^2E) \qquad
\te{(canonically)}.
\end{eqnarray*}
\end{lema}
\vist

\subsection{Cubes with canonical kernels}\label{canonicalkernels}
We introduce here a new subcomplex of $\Z C_*(X)$, consisting of
the \emph{cubes with canonical kernels}. In this new class of
cubes, the transgressions behave almost like a chain morphism.
Namely, if $E$ is an $n$-cube with canonical kernels, we will have
\begin{eqnarray}\label{trivial}
\delta_i^0\tr\nolimits_n(E) &=&\tr\nolimits_{n-1}(\partial_i^1E),\nonumber \\
\delta_i^1\tr\nolimits_n(E) &\cong&
\tr\nolimits_{n-1}(\partial_i^0E)
\oplus\tr\nolimits_{n-1}(\partial_i^2E).
\end{eqnarray}

\begin{df}\label{kernels}
Let $E$ be an $n$-cube. We say that $E$ has \emph{canonical
kernels} if for every   $i=1,\dots,n$ and $\bj\in
\{0,1,2\}^{n-1}$, there is an inclusion
$(\partial_i^0E)^{\bj}\subset (\partial_i^1E)^{\bj}$ of sets and
moreover
 the morphism $$f_i^0:\partial_i^0E\rightarrow \partial_i^1E$$
is the canonical inclusion of cubes.

\end{df}

Let $KC_n(X)\subseteq C_n(X)$ be the subset of all cubes with
canonical kernels. The differential of $\Z C_*(X)$ induces a
differential on $\Z KC_*(X)$ making the inclusion arrow a chain
morphism.

Let $E$ be a $1$-cube i.e. a short exact sequence $E^0
\xrightarrow{f^0} E^1 \xrightarrow{f^1} E^2$. Then, we define
\begin{eqnarray*}
\lambda_1^0(E) &:& 0\rightarrow 0 \rightarrow E^0
\xrightarrow{f^0}
\im f^0 \rightarrow 0, \\
 \lambda_1^1(E) &:& 0\rightarrow \ker f^1 \rightarrow E^1
\xrightarrow{f^1} E^2 \rightarrow 0.
\end{eqnarray*}
Both of them are $1$-cubes with canonical kernels. Then, we define
$$\lambda_1(E)=\lambda_{1}^1(E)-\lambda_1^0(E)\in \Z KC_1(X).$$

For an arbitrary $n$-cube $E\in C_n(X)$ and for every
$i=1,\dots,n$, let $\lambda_i^{0}(E)$ and $\lambda_i^{1}(E)$
 be the $n$-cubes which along the $i$-th direction are:
$$\begin{array}{rclcrcl}
\partial_i^0\lambda_i^{0}(E) &=& 0 & \qquad & \partial_i^0\lambda_i^{1}(E)
&=& \im f_i^0
\\
\partial_i^1\lambda_i^{0}(E) &=& \partial_i^0 E & \qquad & \partial_i^1\lambda_i^{1}(E)
&=& \partial_i^1 E  \\
\partial_i^2\lambda_i^{0}(E) &=& \im f_i^0 & \qquad & \partial_i^2\lambda_i^{1}(E) &=&
\partial_i^2E.
\end{array}$$
Then, we define
\begin{eqnarray*}
\lambda_i(E) & =& -\lambda_i^{0}(E)+\lambda_i^{1}(E),\quad i=1,\dots,n, \\
\lambda(E) & = & \left\{\begin{array}{ll} \lambda_n \cdots
\lambda_1(E),&\textrm{if }n>0, \\ E & \textrm{if }n=0.
\end{array}\right.
\end{eqnarray*}

\begin{prop}\label{deg1} The map
$$\lambda: \Z C_n(X) \rightarrow \Z KC_n(X)$$ is a morphism of
complexes.
\end{prop}
\begin{proof}
First of all, observe that the image by $\lambda$ of any $n$-cube
$E$, is a sum of $n$-cubes with canonical kernels. It is a
consequence of the fact that for any commutative square of
epimorphisms,
$$\xymatrix{ A \ar[r]^{f} \ar[d]_{g} & B \ar[d]^h \\ C \ar[r]_{j} & D,}$$
the set equality $\ker(\ker g \rightarrow \ker h)= \ker(\ker f
\rightarrow \ker j)$ holds. The equality $d\lambda(E)=\lambda
d(E)$ follows from the equalities
\begin{eqnarray*}
\partial_i^j\lambda(E)&=&\partial_i^j(\lambda_n\cdots \lambda_i^1
\cdots \lambda_1 (E))- \partial_i^j(\lambda_n\cdots \lambda_i^0
\cdots \lambda_1 (E)), \\
\partial_i^0(\lambda_n\cdots \lambda_i^1
\cdots \lambda_1 (E)) &=& \partial_i^2(\lambda_n\cdots \lambda_i^0
\cdots \lambda_1 (E)), \\
\partial_i^j(\lambda_n\cdots \lambda_i^1
\cdots \lambda_1 (E)) & =& \lambda_{n-1}\cdots \lambda_1 (\partial_i^jE),\quad j=1,2, \\
\partial_i^1(\lambda_n\cdots \lambda_i^0
\cdots \lambda_1 (E)) & =& \lambda_{n-1}\cdots \lambda_1 (\partial_i^0E).
\end{eqnarray*}
\end{proof}

\subsection{The transgression morphism}\label{transgressionmorphism}

Observe that the face maps $\delta_i^j$ of $\square^{\cdot}$ (as defined in section \ref{projectiucocubical})
induce morphisms on the complex of split cubes
$$\delta_i^j:\Z \Sp\nolimits_{r}(X\times\square^{n})\rightarrow \Z \Sp\nolimits_{r}(X\times\square^{n-1}).$$
 Let $\Z
\Sp^{\square}_{*,*}(X)$ be the 2-iterated chain complex given by
$$\Z \Sp\nolimits^{\square}_{r,n}(X) = \Z \Sp\nolimits_{r}(X\times\square^{n}),$$ and differentials
\begin{eqnarray*}
d &=& d_{\Sp\nolimits_*(X\times\square^{n})}, \\
\delta &=& \sum (-1)^{i+j}\delta^j_i.
\end{eqnarray*}
Let $(\Z \Sp^{\square}_*(X),d_s)$ be the associated simple
complex. Using the transgressions, we define here a morphism of
complexes
$$\Z KC_*(X) \xrightarrow{T} \Z \Sp\nolimits^{\square}_{*}(X)$$ which composed with  $\lambda$
gives the \emph{transgression morphism}
$$\Z C_*(X) \xrightarrow{T} \Z \Sp\nolimits^{\square}_{*}(X).$$
For every $n$-cube $E$ with canonical kernels, the component of
$T(E)$ in $\Z \Spn_{0,n}^{\square}(X)$ is exactly
$(-1)^n\trn_n(E)$. However, the assignment
$$E\mapsto (-1)^n\tr\nolimits_n(E)$$
is not a chain morphism. The failure comes from equality
\eqref{trivial}, since, first of all, the equality holds only up
to some canonical isomorphisms, and second, a direct sum is not a
sum in the complex of cubes. We will add some ``correction cubes''
in $\Z \Spn_{n-m,m}^{\square}(X)$, with $m\neq n$, in order to obtain
a chain morphism $T$.

We start by constructing the morphism step by step in the low
degree cases, deducing from the examples the key ideas.

\textbf{The transgression morphism for $n=1,2$.}
Let $E$ be a $1$-cube with canonical kernels. Then,
$$\delta \trn_1(E)=-\delta_1^0\trn_1(E)+\delta_1^1\trn_1(E)=-E^1+\delta_1^1\trn_1(E).$$
We know that there is a canonical isomorphism (which in this case
is the identity):
$$\delta_1^1\trn_1(E)\cong E^0\oplus E^2.$$
Hence, the differential of
$$T(E):=(-\tr\nolimits_1(E), E^0 \rightarrow \delta_1^1\trn_1(E) \rightarrow E^2)\in \Z C^{\square}_{0,1}(X)
\oplus \Z C^{\square}_{1,0}(X) $$ is exactly $E^1-E^0-E^2\in \Z
C_0(X).$

Let $E$ be a $2$-cube with canonical kernels. Then,
\begin{eqnarray*}
\delta\trn_2(E)&=&-\delta_1^0\trn_2(E)+\delta_1^1\trn_2(E)+\delta_2^0\trn_2(E)-\delta_2^1\trn_2(E)
\\ &=&
-\trn_1(\partial_1^1E)+\delta_1^1\trn_2(E)+\trn_1(\partial_2^1E)-\delta_2^1\trn_2(E).
\end{eqnarray*}
Let
$$T_{1}^i(E) = \tr\nolimits_1(\partial_i^0 E)
\rightarrow  \delta_i^1\tr\nolimits_2(E)  \rightarrow
\tr\nolimits_1(\partial_i^2 E),$$ where the arrows are defined by
the canonical isomorphism $\delta_i^1\trn_2(E)\cong
\trn_1(\partial_i^0E) \oplus \trn_1(\partial_i^2E)$. Let
$$T_2(E)= \quad\begin{array}{c} \xymatrix{ E^{00} \ar[r] \ar[d] &
\delta_1^1\tr\nolimits_1(E^{*0}) \ar[r] \ar[d] & E^{20}
 \ar[d] \\
\delta_1^1 \tr\nolimits_1(E^{0*})  \ar[r] \ar[d] &
\delta_1^1\delta_2^1\tr\nolimits_2(E)  \ar[r] \ar[d] &
\delta_1^1\tr\nolimits_1(E^{2*})
  \ar[d] \\
 E^{02} \ar[r] & \delta_1^1\tr\nolimits_1(E^{*2})  \ar[r] & E^{22},}\end{array}$$
with the arrows induced by the canonical isomorphisms of   lemma
\ref{trans}. Then, for $n=2$, we define
$$T(E):=(\tr\nolimits_2(E), \sum_{i=1,2} (-1)^{i} T_{1}^i(E),
T_{2}(E) )\in \Z C^{\square}_{0,2}(X)\oplus \Z
C^{\square}_{1,1}(X)\oplus \Z C^{\square}_{2,0}(X).$$ By lemma
\ref{trans3}, since $E$ is a cube with canonical kernels, these
cubes are all split. We fix the splittings to be the canonical
isomorphisms of lemma \ref{trans3}. Then, in $\Z
C^{\square}_{0,1}(X) \oplus \Z C^{\square}_{1,0}(X)$,
\begin{eqnarray*}
d_sT(E) &=& (\delta \tr\nolimits_2(E)+ \sum_{i=1,2} (-1)^{i}
dT_{1}^i(E), -\sum_{i=1,2} (-1)^{i}
\delta T_{1}^i(E)+d T_{2}(E)) \\
&=& \sum_{i=1,2}\sum_{j=0,1}(-1)^{i+j} T(\partial_i^jE).
\end{eqnarray*}

\textbf{The transgression morphism.}
Recall that if $\bj\in \{0,1,2\}^{m}$, we  defined in section
\ref{split}
$$s(\bj)=\#\{r|\ j_r=1\},$$ and the multi-index $u(\bj)=(u_1,\dots,u_{s(\bj)})$ with $u_i$ the
indices such that $j_{u_i}=1$ and ordered by $u_1<\dots <
u_{s(\bj)}$. Consider the set of multi-indices
$$J_n^m:=\{\bi=(i_1,\dots,i_{n-m})|\ 1\leq i_1 < \cdots < i_{n-m}\leq n\}.$$
Then, for every $\bi\in J_n^m$ and $\bj\in \{0,1,2\}^{n-m}$, we define
$$\bi(\bj)=(i_{u_1},i_{u_2}-1,\dots,i_{u_l}-l+1,\dots,i_{u_{s(\bj)}}-s(\bj)+1).$$

\begin{df}
Let $E$ be an $n$-cube with canonical kernels. For every $0\leq m
\leq n$ and $\bi\in J_n^m$, we define $T_{n-m,m}^{\bi}(E)\in \Z
C^{\square}_{n-m,m}(X)$ as the $(n-m)$-cube on $X\times \square^m$
given by:
\begin{enumerate*}[$\blacktriangleright$]
\item If $\bj\in \{0,2\}^{n-m}$  then
$$T_{n-m,m}^{\bi}(E)^{\bj}:=\tr\nolimits_m^{\bi}(E)^{\bj}=\trn_m(\partial_{\bi}^{\bj}E).$$
\item If $\bj\in \{0,1,2\}^{n-m}$ with $j_k= 1$ for some $k$, then
we define
$$T_{n-m,m}^{\bi}(E)^{\bj}:=
\big(\delta_{\bi(\bj)}^{\bun}\tr\nolimits_{m+s(\bj)}^{\partial_{u(\bj)}(\bi)}(E)\big)^{\partial_{u(\bj)}(\bj)}.$$
\end{enumerate*}
That is, we start by considering the first transgression of $E$,
iteratively, in the directions $s$ not in the multi-index $\bi$
and in those $i_s$ with $j_s= 1$. This gives a $(n-m-s(\bj))$-cube
on $X\times \square^{m+s(\bj)}$. Then, we apply $\delta^1_s$ for
each affine component coming from a direction with $j_s=1$. We
obtain a $(n-m)$-cube on $X\times \square^{m}$.
\end{df}
Observe that in the above definition, the second case
generalizes the first case.

\begin{lema}\label{transgressionsplit} For every $n$-cube $E$ with canonical kernels and every $\bi \in J_n^m$, the
cube $T_{n-m,m}^{\bi}(E)$ is split.
 \end{lema}
\begin{proof}
 If $\bj\in \{0,1,2\}^{n-m}$, it follows by lemma \ref{trans} that,
$$T^{\bi}_{n-m,m}(E)^{\bj}=\big(\delta_{\bi(\bj)}^{\bun}\tr\nolimits_{m+s(\bj)}^{\partial_{u(\bj)}(\bi)}(E)\big)^
{\partial_{u(\bj)}(\bj)} \cong \bigoplus_{\bm\in\{0,2\}^{s(\bj)}}
\tr\nolimits_{m}^{\bi}(E)^{\sigma_{u(\bj)}^{\bm}(\bj)}.
$$
\end{proof}
Observe that  for any morphism of $n$-cubes $E\rightarrow F$, there is a
commutative square
$$ \xymatrix{
T^{\bi}_{n-m,m}(E)^{\bj} \ar[r] \ar[d]_{\cong} & T^{\bi}_{n-m,m}(F)^{\bj} \ar[d]^{\cong} \\
\bigoplus\limits_{\bm\in\{0,2\}^{s(\bj)}}
\tr\nolimits_{m}^{\bi}(E)^{\sigma_{u(\bj)}^{\bm}(\bj)} \ar[r] &
\bigoplus\limits_{\bm\in\{0,2\}^{s(\bj)}}
\tr\nolimits_{m}^{\bi}(F)^{\sigma_{u(\bj)}^{\bm}(\bj)}. }$$

Finally, we consider
\begin{equation}
T_{n-m,m}(E) := \sum_{\bi\in J^{n-m}_n}(-1)^{|\bi|+\sigma(n-m)+m}
T_{n-m,m}^{\bi}(E)\in \Z  \Spn^{\square}_{n-m,m}(X),\end{equation}
where $\sigma(k)=1+\cdots + k= \frac{k(k+1)}{2}$.
\begin{prop}\label{trans2}
The map
\begin{eqnarray*}
\Z KC_n(X) & \xrightarrow{T} & \bigoplus_{m=0}^n \Z \Spn^{\square}_{n-m,m}(X) \\
E &\mapsto & T_{n-m,m}(E)
\end{eqnarray*}
is a chain morphism.
\end{prop}
\begin{proof}
We have to see that  $T$ commutes with the differentials. Remember
that the differential $d_s$ in the simple complex  is defined, in
the $(n-m,m)$-component, by $d + (-1)^{n-m} \delta$. Therefore, we
have to see that for every $m=0,\dots,n-1$, the equality
\begin{equation}\label{eq5}
d T_{n-m,m}+(-1)^{n-m-1}\delta T_{n-m-1,m+1} = T_{n-m-1,m}d
\end{equation}
holds. The right hand side is
\begin{eqnarray*}
T_{n-m-1,m}d&=& \sum_{r=1}^n \sum_{j=0}^2(-1)^{r+j} T_{n-m-1,m}\partial_r^j\\
&=& \sum_{r=1}^n \sum_{j=0}^2(-1)^{r+j} \sum_{\bi\in
J^{n-m-1}_{n-1}}(-1)^{|\bi|+\sigma(n-m-1)+m}
T_{n-m-1,m}^{\bi}\partial_r^j.
\end{eqnarray*}
We compute the terms $d$ and $\delta$ on the left hand side separately.
\begin{eqnarray*}
dT_{n-m,m}  &=& \sum_{r=1}^{n-m}\sum_{\bi\in
J^{n-m}_n}(-1)^{|\bi|+\sigma(n-m)+r+m} (\partial_r^0-\partial_r^1
+
\partial_r^2 )T_{n-m,m}^{\bi}.
\end{eqnarray*}
By definition,
$\partial_r^lT_{n-m,m}^{\bi}=T_{n-m-1,m}^{\partial_r(\bi)-\bun_{r+1}^{n-m}}\partial_{i_r}^l$
if $l=0,2$.

Since $(-1)^{\sigma(n-m)+n-m}=(-1)^{\sigma(n-m-1)}$, we obtain
\begin{eqnarray*}
dT_{n-m,m}
&=&  \sum_{r=1}^n (-1)^r [T_{n-m-1,m}\partial_{i_r}^0+
T_{n-m-1,m}\partial_{i_r}^2]- \\ &&  - \sum_{\bi\in
J^{n-m}_{n}}\sum_{r=1}^{n} (-1)^{|\bi|+r+\sigma(n-m)+m}
\partial_r^1 T_{n-m,m}^{\bi}.
\end{eqnarray*}
For the summand corresponding to the differential $\delta$, one
has that $$(-1)^m\delta T_{n-m-1,m+1} = (1)+(2), $$ with
\begin{eqnarray*}
(1) &=& -\sum_{r=1}^m \sum_{\bi\in
J^{n-m-1}_n}(-1)^{|\bi|+\sigma(n-m-1)+r}
\delta_r^0T_{n-m-1,m+1}^{\bi},
 \\
(2) &=& \sum_{r=1}^m \sum_{\bi\in
J^{n-m-1}_n}(-1)^{|\bi|+\sigma(n-m-1)+r} \delta_r^1
T_{n-m-1,m+1}^{\bi}.
\end{eqnarray*}
For every $\bj\in \{0,1,2\}^{n-m-1}$, we obtain by definition
$$(\delta_r^0 T_{n-m,m}^{\bi})^{\bj} =
\left(\delta_r^0\delta_{\bi(\bj)}^{\bun}\tr\nolimits_{m+s(\bj)}^{\partial_{u(\bj)}(\bi)}\right)^
{\partial_{u(\bj)}(\bj)}.$$
Using the commutation rules of the faces $\delta_*^*$ and $\partial_*^*$, we
obtain that
$$\delta_r^0 T_{n-m,m}^{\bi} =
T_{n-m-1,m}^{\bi-\bun_{l+1}^{n-m-1}}\partial_{r+l}^1,$$ where $l$ is the
maximal between the indices $k$, such that $r+k-1\geq i_k$. Therefore,
\begin{eqnarray*}
(1)&=& \sum_{r=1}^m \sum_{l=0}^{n-m-1}\sum_{\substack{\bi\in J^{n-m-1}_n \\
l=\max\{k|r+k-1\geq i_k\}}}(-1)^{|\bi|
+\sigma(n-m-1)+r+1} T_{n-m-1,m}^{\bi-\bun_{l+1}^{n-m-1}}\partial_{r+l}^1\\
&=&\sum_{s=1}^n \sum_{l=0}^{n-m-1}\sum_{\substack{\bi\in J^{n-m-1}_n \\
l=\max\{k|s-l+k-1\geq i_k\}}}(-1)^{|\bi|
+\sigma(n-m-1)+1+s-l} T_{n-m-1,m}^{\bi-\bun_{l+1}^{n-m-1}}\partial_{s}^1  \\
&=&\sum_{s=1}^n \sum_{\bi\in J^{n-m-1}_{n-1}
}(-1)^{|\bi|+\sigma(n-m-1)+n-m+s} T_{n-m-1,m}^{\bi}\partial_{s}^1
\\ &=& (-1)^{n+1}\sum_{s=1}^n
(-1)^{s+1}T_{n-m-1,m}\partial_{s}^1.
\end{eqnarray*}
All that remains is to see that
$$(2)= (-1)^{n-m-1}\sum_{\bi\in
J^{n-m}_{n}}\sum_{r=1}^{n} (-1)^{|\bi|+r+\sigma(n-m)}
\partial_r^1 T_{n-m,m}^{\bi} =:(*).$$
Recall that
$$(\partial_r^1 T_{n-m,m}^{\bi})^{\bj}=(T^{\bi}_{n-m,m})^{s_r^{1}(\bj)}=
\left(\delta_{\bi(s_r^1\bj)}^{\bun}\tr\nolimits_{m+s(\bj)+1}^{\partial_{u(\bj)}\partial_r(\bi)}
\right)^{\partial_{u(\bj)}(\bj)}.$$
An easy calculation shows that
$\delta_{\bi(s_r^1\bj)}^{\bun}=\delta_{\bi(\bj)}^1\delta_{i_r-r+1}^{\bun}.$ Therefore,
$$\partial_r^1 T_{n-m,m}^{\bi} =
\delta_{i_r-r+1}^1 T^{\partial_r\bi}_{n-m,m} $$ and hence,
\begin{eqnarray*}
(*) &=& \sum_{\bi\in J^{n-m}_{n}}\sum_{r=1}^{n} (-1)^{|\bi|+r+\sigma(n-m)}
\partial_r^1 T_{n-m,m}^{\bi} \\
&=& \sum_{\bi\in J^{n-m}_{n}}\sum_{r=1}^{n} (-1)^{|\bi|+r+\sigma(n-m)}
\delta_{i_r-r+1}^1T^{\partial_r\bi}_{n-m,m} \\ &=& \sum_{\bi\in
J^{n-m-1}_{n}}\sum_{r=1}^{n} (-1)^{|\bi|+r+1+\sigma(n-m)}
\delta_r^1T^{\partial_r\bi}_{n-m,m} \\ &=&
 \sum_{\bi\in
J^{n-m-1}_{n}}\sum_{r=1}^{n} (-1)^{(n-m-1)+|\bi|+r+\sigma(n-m-1)}
\delta_r^1T^{\partial_r\bi}_{n-m,m}= (2),
\end{eqnarray*}
and the proposition is proved.
\end{proof}

\section{Adams operations on rational algebraic K-theory}
\index{Adams operations} To sum up, we have defined the following
chain morphisms:
\begin{enumerate*}[$\blacktriangleright$]
\item In section \ref{transgression}, we defined a morphism
$$\Z C_*(X) \xrightarrow{T} \Z \Spn^{\square}_*(X).$$
That is,  a collection of split cubes on $X\times \square^*$ is
assigned to  every $n$-cube. \item  In section \ref{adamssplit},
for every $k\geq 1$, we defined a chain morphism
$$\Psi^k: \Z \Spn_*(X) \rightarrow \Z C_*(X).$$
\end{enumerate*}

Since the maps $\Psi^k$ and $T$ are functorial, they are natural
on schemes $X$ in $\mathcal{C}_B$. Therefore, there is an induced
map of 2-iterated chain complexes
$$\Psi^k: \Z \Spn^{\square}_{*,*}(X)\rightarrow \Z C^{\square}_{*,*}(X)
\rightarrow \wZ C_{*,*}^{\widetilde{\square}}(X),$$
which induces a morphism on the associated simple complexes
$$ \Psi^k: \Z \Spn^{\square}_{*}(X)\rightarrow \Z C^{\square}_{*}(X)\rightarrow \wZ C_{*}^{\widetilde{\square}}(X).$$
The composition with the morphism $T$ gives a morphism
$$\Psi^k: \Z C_*(X) \rightarrow \Z C^{\square}_*(X)\rightarrow \wZ C^{\widetilde{\square}}_*(X)
\cong N C^{\widetilde{\square}}_*(X).$$ Finally, considering the normalized
complex of cubes of lemma \ref{N}, we define
$$\Psi^k: NC_*(X) \hookrightarrow
\Z C_*(X) \rightarrow NC^{\widetilde{\square}}_*(X).$$

\begin{cor}\label{adamsalgebraic} For every $k\geq 1$, there is a well-defined chain morphism
$$\Psi^k: NC_*(X) \rightarrow  N C^{\widetilde{\square}}_*(X).$$
If $X$ is a regular noetherian scheme, then for every $n$, there are induced
morphisms
$$\Psi^k: K_n(X)\otimes \Q \rightarrow K_n(X)\otimes \Q.$$
\end{cor}
\begin{proof}
It is a consequence of lemmas \ref{cubes1} and \ref{dold}.
\end{proof}

\begin{theo}\label{myadamsgood} Let $X$ be a regular noetherian scheme of finite Krull dimension.
For every $k\geq 1$ and $n\geq 0$, the morphisms
$$\Psi^k: K_n(X)\otimes \Q \rightarrow K_n(X)\otimes \Q$$
of corollary \ref{adamsalgebraic}, agree with the Adams operations defined by
Gillet and Soul\'e in \cite{GilletSouleFiltrations}.
\end{theo}
\begin{proof} We have constructed a functorial morphism, at the level of chain
complexes, which by definition can be extended to simplicial schemes. Moreover,
 they induce the usual Adams operations on the $K_0$-groups, i.e.
the Adams operations derived from the lambda structure coming from the exterior
product of locally free sheaves.  Then, the statement follows from corollary 4.4
 in \cite{Feliu1}.
\end{proof}

\begin{cor}
The Adams operations defined here satisfy the usual identities for any finite
dimensional regular noetherian scheme.
\end{cor}

\providecommand{\bysame}{\leavevmode\hbox to3em{\hrulefill}\thinspace}
\providecommand{\MR}{\relax\ifhmode\unskip\space\fi MR }
\providecommand{\MRhref}[2]{%
  \href{http://www.ams.org/mathscinet-getitem?mr=#1}{#2}
}
\providecommand{\href}[2]{#2}

\end{document}